\documentclass{siamonline190516}

\usepackage{amssymb} 
\usepackage{amsmath}
\usepackage{amsfonts}
\usepackage{mathtools}
\usepackage{graphicx}
\usepackage{color}
\usepackage{epstopdf}
\usepackage{algorithmic}
\Crefname{ALC@unique}{Line}{Lines}
\ifpdf
  \DeclareGraphicsExtensions{.eps,.pdf,.png,.jpg}
\else
  \DeclareGraphicsExtensions{.eps}
\fi

\usepackage{enumitem}
\setlist[enumerate]{leftmargin=.5in}
\setlist[itemize]{leftmargin=.5in}

\makeatletter
\newcommand{\RR}{\mathbb R}
\newcommand{\FE}{finite element}

\newcommand{\iid}{i.i.d.}
\DeclareMathOperator{\ML}{ML}
\DeclareMathOperator{\MC}{MC}
\DeclareMathOperator{\LF}{LF}
\DeclareMathOperator{\LTSLF}{LF-LTS}
\DeclareMathOperator{\fine}{fine}
\DeclareMathOperator{\coar}{coarse}

\usepackage{amsopn}

\newsiamremark{remark}{Remark}
\def\bysame{\leavevmode\hbox to3em{\hrulefill}\thinspace}

\headers{Uncertainty Quantification by MLMC and LTS for Wave Propagation}{M. J. Grote, S. Michel, and F. Nobile}

\title{Uncertainty Quantification by MLMC and Local Time-stepping \\ For Wave Propagation
}

\author{
Marcus J. Grote\thanks{Department of Mathematics and Computer Science,
University of Basel, Spiegelgasse 1, 4051 Basel, Switzerland
(\email{marcus.grote@unibas.ch}, \email{simon.michel@unibas.ch})}
\and Simon Michel\footnotemark[2]
\and Fabio Nobile\thanks{Calcul Scientifique et Quantification de l'Incertitude (CSQI), Institute of Mathematics, \'{E}cole Polytechnique F\'{e}d\'{e}rale de Lausanne, 1015 Lausanne, Switzerland (\email{fabio.nobile@epfl.ch})}
}

\usepackage{soul}
\soulregister\cite{7}
\soulregister\ref{7}
\soulregister\eqref{7}

\begin{document}
\maketitle

\begin{abstract}
Because of their robustness, efficiency and non-intrusiveness,
Monte Carlo methods are probably the most popular approach in uncertainty quantification to computing expected values of quantities of interest (QoIs). Multilevel Monte Carlo (MLMC) methods significantly reduce the computational cost by
distributing the sampling across a hierarchy of discretizations and allocating most samples to the coarser grids. 
For time dependent problems, spatial coarsening typically entails an increased time-step. 
Geometric constraints, however, may impede uniform coarsening thereby forcing 
some elements to remain small across all levels. 
If explicit time-stepping is used, the time-step will then be dictated by the smallest element on each level for numerical stability.
Hence, the increasingly stringent CFL condition on the time-step on coarser levels significantly reduces the advantages of the multilevel approach.
To overcome that bottleneck we propose to combine the multilevel approach of MLMC with local time-stepping (LTS). By adapting the time-step to the locally refined elements on each level,  
the efficiency of MLMC methods is restored even in the presence of complex geometry without sacrificing the explicitness and inherent parallelism.
In a careful cost comparison, we quantify the reduction in computational cost for local refinement either inside a small fixed region or towards a reentrant corner.
\end{abstract}

\begin{keywords}
  Uncertainty quantification,
  Multilevel Monte Carlo, 
  wave propagation,
  finite element methods,
  local time-stepping,
  explicit time integration.
\end{keywords}

\begin{AMS}
  65C05, 65L06, 65M20, 65M60, 65M75.
\end{AMS}

\section{Introduction}

Mathematical models based on partial differential equations
  (PDE) are widely used to describe complex phenomena and to make
  predictions in real-world applications. All such mathematical models,
  however, are affected by a certain degree of uncertainty that may
  arise because of imperfect characterization or intrinsic variability
  of model parameters, constitutive laws, forcing terms, initial
  states, etc.  
Uncertainty is typically included
  in PDE based models by replacing input parameters by stochastic
  variables or processes. 
  That uncertainty is then propagated across
space (and possibly time) by the solution $u$ of the corresponding
stochastic PDE and thereby determines the uncertainty in any observed
quantity of interest (QoI) $Q[u]$. 

In mathematical models from acoustics, electromagnetics or
elasticity, waves, as
ubiquitous information carriers, will also propagate uncertainty about input parameters
over long distances with little regularizing or smoothing effects.
The inherent lack of regularity hampers  the use of computational uncertainty quantification methods, such as polynomial chaos expansions or sparse quadratures, which rely on the smoothness of the so-called \emph{parameter-to-solution} map and/or  on the low dimensionality of the input space \cite{TryoenJCP2010,motamed.nobile.tempone:stochastic, motamed.nobile.tempone:analysis}. 

In contrast,
Monte Carlo (MC) methods \cite{Fishman}, probably the most popular alternative to quantifying the uncertainty in any QoI, are robust to the dimension of the input parameters and the lack of regularity of the parameter-to-QoI map. 
By drawing independent realizations from the input probability distribution on a sample space $\Omega$, they compute for each sample $\omega\in\Omega$ a solution 
$u(\omega,.)$ of the forward problem and thus permit to estimate the statistics of the QoI $Q[u(\omega,.)]$. Although
MC methods are easy to implement, their convergence in the number of samples is rather slow.

Multilevel Monte Carlo (MLMC) methods, first introduced for applications in parametric integration by Heinrich  \cite{Heinrich1998,Heinrich1999} and later extended by Giles in his seminal paper \cite{Giles08} to multi-level approximations of stochastic differential equations, significantly reduce the computational cost by
distributing the sampling across a hierarchy of discretizations and computing 
most samples on coarser grids. 
In recent years, MLMC methods thus have proved extremely efficient, 
versatile and robust for uncertainty quantification (UQ) in a wide range of 
problems governed by stochastic PDEs, including, to name just a few, 
elliptic equations with random coefficients \cite{CliffeGiles11,TSGU13,BarthSchwabZollinger2011}, 
parabolic PDEs \cite{Bart-Lang-Schwab}, conservation laws and 
compressible aerodynamics \cite{Mishra-Schwab MCOM 2012, Mishra-Schwab-Sukys-1, Pisaroni-Nobile-Leyland CMAME}, 
acoustic and seismic wave propagation \cite{Mishra-Schwab-Sukys-2,  Ballesio-Beck-Pandey-Parisi-vonSchwerin-Tempone}, 
obstacle problems \cite{Bierig-Chernov}, and multiscale problems \cite{Abdulle-Barth-Schwab}.

For time dependent problems, spatial coarsening on higher levels usually entails a larger time-step,
thereby reducing even further the computational cost of individual sample solutions $u(t,x; \omega)$.  
Geometric constraints or singularities, however, may impede uniform coarsening, thus forcing 
some elements in the mesh to remain small across all levels. 
If explicit time-stepping is used, the time-step will then be dictated by the smallest element on each level due to the CFL stability condition. 
Hence, standard explicit time-stepping schemes will become increasingly inefficient on coarser levels due to the ever more restrictive CFL condition. 

To overcome the increasingly stringent bottleneck across all levels,  we propose to use local time-stepping (LTS)  \cite{DiazGrote09,GroteMehlinMitkova15}
methods for the time integration on each level. By using a smaller time-step only inside the locally refined region, LTS methods thus permit to greatly improve the efficiency of MLMC methods even in the presence of complex geometry without sacrificing the explicitness and inherent parallelism.
We carefully analyze the computational cost of the MLMC algorithm in the
presence of locally refined meshes, first with a standard explicit time-stepping method and then with a local time-stepping method. In doing so, we differentiate between local refinement inside a small fixed region or towards a reentrant corner. 
In the former case, we quantify the gain of using LTS over standard time-stepping methods, whereas in 
the latter we prove that LTS even improves the asymptotic complexity.
In a series of numerical experiments, we illustrate the significant gain over standard time-stepping obtained by using
LTS methods for the time integration at all levels.

The rest of our paper is structured as follows. In Section 2, we recall the standard MLMC Algorithm \cite{Giles08}  
to estimate a generic QoI in any (finite or infinite dimensional) Hilbert space; for instance, the QoI may be a 
functional of $u$, or the solution itself. 
In Section 3, we 
present the cost comparison between MLMC with or without local time-stepping.
Finally, in Section 4, 
we present 
numerical experiments including an example with complex geometry in two space dimensions. 

\newpage
\section{Multilevel Monte Carlo method for wave equations with random coefficients}

\subsection{Model problem}

We consider  the wave equation with stochastic coefficient in a bounded domain $D \subset \RR^d$,
\begin{equation}
\left\{
\begin{aligned}
\frac{\partial^2}{\partial t^2} u(\mathbf{x},t,\omega) - \nabla \cdot \left(c^2(\mathbf{x},\omega) \; \nabla u(\mathbf{x},t,\omega)\right) &= f(\mathbf{x},t) & & \forall \; \mathbf{x} \in D, \; t \in (0,T], \; \omega \in \Omega, \\
u(\mathbf{x},0,\omega) &= u_0(\mathbf{x}) & & \forall \; \mathbf{x} \in D, \; \omega \in \Omega,  \\
\frac{\partial}{\partial t} u(\mathbf{x},0,\omega) &= v_0(\mathbf{x}) & & \forall \; \mathbf{x} \in D, \; \omega \in \Omega 
\end{aligned}
\right.
\label{eq:Wave_hd}
\end{equation}
with appropriate (deterministic) boundary conditions. 
Here, we model the uncertainty in the wave speed $c>0$ as a time independent random field $c : D \times \Omega\rightarrow \RR$, where $\Omega$ is the sample space of a complete probability space.
More general models that also allow randomness in the geometry, the forcing term, or an acceleration term could also be considered.
We are interested in estimating the expected value $\mathbb{E}[Q]$ of some quantity of interest (QoI) $Q: \omega \mapsto Q(\omega) = Q\left(u\left(\cdot,\cdot,\omega\right)\right) \in V$ related to the solution $u = u(\mathbf{x},t,\omega)$.

We consider a generic case where $V$ is a Hilbert space.
For instance, if $Q$ is the value of some functional of $u$, we simply set $V = \RR$ (cf. \cite{Appelo,TSGU13}).
On the other hand, if the QoI is the (weak) solution itself at a fixed time $T>0$, $Q(\omega) :=  u(\cdot,T,\omega) \in V=H^1(D)$ $\forall\, \omega\in\Omega$, we let $V$ correspond to the solution space.
This setting can easily be generalized to the (weaker) case when $Q(\omega) \in V$ for almost every $\omega\in\Omega$.

\subsection{Construction of the MLMC method}
To derive the MLMC approximation for \eqref{eq:Wave_hd}, let $Q_h$ denote a (numerical) \FE\ approximation to the QoI $Q = Q(u)$, with $h$ the discrete mesh size. 
To estimate $\mathbb{E}[Q]$, one computes  approximations or estimators $\widehat{Q}_h$ to $\mathbb{E}[Q_h]$. 
The accuracy of the approximations is quantified by the mean square error (MSE)
\begin{equation}
e\left(\widehat{Q}_h\right)^2 := \mathbb{E}\left[\left\|\widehat{Q}_h - \mathbb{E}[Q]\right\|_V^2\right].
\label{eq:errordef}
\end{equation}

The main idea of the MLMC method is to sample the QoI $Q$ from several approximations $Q_\ell := Q_{H_\ell}$ on a sequence of discretizations $\ell = 0,1,\ldots,L$.
Then, each level uses its individual mesh size $H_\ell = H_0 / 2^\ell$ in space and time-step $\Delta t_\ell$ in time, where the latter must satisfy a standard CFL condition, $\Delta t_\ell \leq C H_\ell$, for numerical stability, if explicit time-stepping is used. 

For the approximate solution on the finest level with mesh size $H_L$, it holds that
\begin{equation*}
\mathbb{E}[Q_L] = \mathbb{E}[Q_0] + \sum\limits_{\ell = 1}^{L} \mathbb{E} [Q_\ell - Q_{\ell - 1}] 
= \sum\limits_{\ell = 0}^{L} \mathbb{E}\left[\Delta Q_{\ell}\right]
\end{equation*}
with
\begin{equation*}
 \Delta Q_{\ell} := \left\{
\begin{array}{ll}
Q_{0}, & \ell = 0, \\
Q_{\ell} - Q_{{\ell-1}}, & \ell = 1,\ldots,L, 
\end{array}
\right.
\end{equation*}
random variables on $\Omega$.
This motivates the MLMC estimator of $\mathbb{E}[Q_L]$,
\begin{align}
\widehat{Q}^{\ML}_h : & &\Omega^M &\to V \nonumber\\
& &\hat{\omega} := \left\{ \omega^{(i,\ell)} \mid i = 1,\ldots,N_\ell, \, \ell = 0,\ldots,L \right\} &\mapsto \sum\limits_{\ell = 0}^{L} {\frac{1}{N_\ell} \sum\limits_{i = 1}^{N_\ell} {\left(\Delta Q_{\ell}\left(\omega^{(i,\ell)}\right)\right)}},
\label{eq:MLMCestimator}
\end{align}
where  
$M = \sum_\ell N_\ell$ with  $N_\ell$ denoting the size of the sample on each level $\ell$
and where the probability measure on $\Omega^M$ is the tensor product measure $\mathbb{P}^{\bigotimes M}$, i.e. $\hat{\omega}$ is an independent sample.
For efficiency of the estimator, it is crucial to judiciously choose the parameter values $N_\ell$ and $L$. 
If we set $\mu = \mathbb{E}[Q] \in V$ in  \eqref{eq:errordef}, we obtain
\begin{equation*}
\begin{aligned}
e\left(\widehat{Q}_h^{\ML}\right)^2 &= \mathbb{E}\left[\left<\widehat{Q}_h^{\ML} - \mu,\widehat{Q}_h^{ML} - \mu\right>_V\right] \\
&= \mathbb{E}\left[\left< \widehat{Q}_h^{\ML} - \mathbb{E}\left[Q_{L}\right] + \mathbb{E}\left[Q_{L}\right] - \mu,\widehat{Q}_h^{ML} - \mathbb{E}\left[Q_{L}\right] + \mathbb{E}\left[Q_{L}\right] - \mu \right>_V\right], \\
&= \mathbb{E}\left[\left\| \widehat{Q}_h^{\ML} - \mathbb{E}\left[Q_{L}\right] \right\|_V^2\right] + \left\| \mathbb{E}[Q_{L}] - \mu \right\|_V^2 + 2\,\mathbb{E}\left[\left< \widehat{Q}_h^{\ML} - \mathbb{E}\left[Q_{L}\right], \mathbb{E}[Q_{L}] - \mu \right>_V \right],
\end{aligned}
\end{equation*}
where 
the expectation is understood with respect to the product probability measure $\mathbb{P}^{\bigotimes M}$ and
the last term equals to zero as $\mathbb{E}[\widehat{Q}_h^{\ML}] = \mu$. 
For the first term in the last equation, we now insert the definition of the MLMC estimator (\ref{eq:MLMCestimator}), 
where we denote by $Q^{(i,\ell)}$, $i = 1,\ldots,N_\ell$, the independent, identically distributed (\iid) ``copies'' of the random variable $Q$ on level $\ell$.
Thus, we obtain
\begin{equation*}
\begin{aligned}
e\left(\widehat{Q}_h^{\ML}\right)^2 &= \mathbb{E}\left[\left\| \sum_{\ell=0}^L N_ \ell^{-1} \sum_{i=1}^{N_\ell} \left(\Delta Q_\ell^{(i,\ell)} - \mathbb{E}\left[\Delta Q_\ell\right]\right)  \right\|_V^2\right] + \left\| \mathbb{E}[Q_{L} - Q] \right\|_V^2 \\
&= \sum_{\ell=0}^{L} N_\ell^{-1} \mathbb{E}\left[\left\| \Delta Q_\ell - \mathbb{E}\left[\Delta Q_\ell\right] \right\|_V^2 \right] + \left\| \mathbb{E}[Q_{L} - Q] \right\|_V^2.
\end{aligned}
\end{equation*}
Note that the last step is due to the fact that $\Delta Q_\ell^{(i,\ell)}$ are 
indeed \iid
Let
\begin{equation}
V_\ell = \mathbb{E}\left[\left\| \Delta Q_\ell - \mathbb{E}\left[\Delta Q_\ell\right] \right\|_V^2 \right] 
= \mathbb{E}\left[\left\| \Delta Q_\ell \right\|_V^2 \right] - \left\|\mathbb{E}\left[\Delta Q_\ell\right] \right\|_V^2
\label{eq:Var_Def}
\end{equation}
denote the variance on a level $\ell = 0,1,\ldots,L$. 
Then, the mean square error can be split as
\begin{equation}
e\left(\widehat{Q}_h^{\ML}\right)^2 = \sum_{\ell=0}^{L} N_\ell^{-1} V_\ell + \left\| \mathbb{E}\left[Q_{L} - Q\right] \right\|_V^2,
\label{eq:MSE_splitform}
\end{equation}
where the first term is interpreted as the stochastic error or total variance of the estimator and the second term as the numerical error or bias term. 

One may now want to equilibrate those two parts. This means that for any given root MSE tolerance $\varepsilon$, we want to choose the number of refinement levels $L$ and number of samples $N_\ell$ such that both error contributions are bounded by $\varepsilon^2/2$. 
Note, however, that splitting the error equally is neither necessary nor optimal \cite{NobileTempone2015} and is therefore only a simplification.

Let $C_\ell$ denote the cost of computing a single sample $\Delta Q_\ell^{(i,\ell)}$. 
As a consequence, the total cost for computing the MLMC estimator is then given by the sum 
\begin{equation}
\mathcal{C}\left[\widehat{Q}^{\ML}_h\right] = \sum\limits_{\ell=0}^{L} N_\ell C_\ell.
\label{eq:MLMC_totalcost_def}
\end{equation}
If we assume $L$ and the overall cost to be fixed, the optimal number of samples $N_\ell$ is determined by minimizing the total variance $\sum_{\ell=0}^{L} N_\ell^{-1} V_\ell$, which yields the lower bound
\begin{equation}
N_\ell \geq \dfrac{2}{\varepsilon^2} \sqrt{\dfrac{V_\ell}{C_\ell}  } \sum\limits_{\ell^\prime=0}^{L} \sqrt{V_{\ell^\prime} C_{\ell^\prime} }.
\label{eq:optimal_Nl}
\end{equation}

When estimates for the numerical part of the error (\ref{eq:MSE_splitform}) are explicitly known a priori, it is possible to derive theoretically optimal choices for $L$ \cite{BarthSchwabZollinger2011}. 
In general, however, those constants are not known a priori.
Therefore, we instead opt for the approach as in \cite{CliffeGiles11}, which chooses $L$ ``on-the-fly''.

\begin{algorithm}
\caption{Multilevel Monte-Carlo}
\label{alg:MLMCGiles}
\begin{algorithmic}[1]
\STATE\label{alg:MLMCGiles_step1}{ 
Initialize $L = 2$ and set initial values for $N_\ell$ on levels $\ell=0,1,2$.
}
\WHILE{$N_\ell$ was increased previously in \cref{alg:MLMCGiles_step1}, \cref{alg:MLMCGiles_step2c} or \cref{alg:MLMCGiles_step2d} for any $\ell$}
\STATE\label{alg:MLMCGiles_step2a}{
Compute remaining $Q_\ell(w^{(i,\ell)})$ and $Q_{\ell-1}\left(\omega^{(i,\ell)}\right)$ for $i=1,\ldots,N_\ell$ on each level $\ell$.
} 
\STATE\label{alg:MLMCGiles_step2b}{
Compute $\widehat{Q}^{\ML}_{h}$ according to (\ref{eq:MLMCestimator}) and update estimates for $V_\ell$, $\ell=0,1,\ldots,L$.
} 
\STATE\label{alg:MLMCGiles_step2c}{
Update $N_\ell$ for $\ell=0,1,\ldots,L$ according to (\ref{eq:optimal_Nl}) using the new estimates for $V_\ell$.
} 
\IF{test for convergence of the bias term fails,}
\STATE\label{alg:MLMCGiles_step2d}{set L := L + 1 and initialize $N_L$. }
\ENDIF
\ENDWHILE
\end{algorithmic}
\end{algorithm}

In \cref{alg:MLMCGiles_step2b} in the above algorithm, the variances $V_\ell$ are estimated according to (\ref{eq:Var_Def}) by approximating,
\begin{equation}
V_\ell 
=  \mathbb{E}\left[\left\| \Delta Q_\ell \right\|_V^2 \right] - \left\|\mathbb{E}\left[\Delta Q_\ell\right] \right\|_V^2 
\approx \frac{1}{N_\ell-1} \left( \sum_{i=1}^{N_\ell} {\left\| \Delta Q_\ell^{(i,\ell)} \right\|_V^2} - \frac{1}{N_\ell} \left\| \sum_{i=1}^{N_\ell} \Delta Q_\ell^{(i,\ell)} \right\|_V^2\right).
\label{eq:Vell_est}
\end{equation}
In \cref{alg:MLMCGiles_step2d}, we test for convergence by verifying, if
\[
\left\| \mathbb{E}[Q_{L} - Q] \right\|_V^2 < \varepsilon^2/2
\] 
is satisfied for the root MSE tolerance $\varepsilon$. 
Since $\mathbb{E}[Q]$ depends on the (a priori unknown) exact solution $u$, 
we approximate the remaining error from previous levels. 
If we assume that $\left\| \mathbb{E}[Q_{\ell} - Q_{{\ell-1}}] \right\|_V^2 = \mathcal{O}(2^{-\alpha \ell})$ for $\ell \rightarrow \infty$ with $\alpha \geq 1$, we obtain
\begin{equation}
\mathbb{E}[Q - Q_{L}] = \sum\limits_{\ell=L+1}^{\infty} \mathbb{E}[Q_{\ell} - Q_{{\ell-1}}] \simeq \frac{\mathbb{E}[Q_{L} - Q_{{L-1}}]}{2^\alpha - 1},
\label{eq:bias_est}
\end{equation}
where we use the symbol ``$\simeq$'' in the following sense:
\[
A \simeq B \Longleftrightarrow c B \leq A \leq \widehat{c} B, \qquad c,\widehat{c} > 0.
\]

The total cost for computing the MLMC estimate $\widehat{Q}^{\ML}_h$ in a Hilbert space is characterized by the following theorem, similar to \cite[Theorem 1]{CliffeGiles11}.

\begin{theorem}
Suppose there exist constants $\alpha, \beta, \gamma > 0$, such that $\alpha \geq \frac{1}{2} \min(\beta,\gamma)$,
\begin{itemize}
\item $\left\| \mathbb{E}[Q_{h} - Q] \right\|_V \leq \mathcal{O}\left(h^{\alpha}\right)$ as $h \rightarrow 0$,
\item $V_\ell \leq \mathcal{O}\left(H_\ell^\beta\right)$ and
\item $C_\ell \leq \mathcal{O}\left(H_\ell^{-\gamma}\right)$ as $\ell \rightarrow \infty$.
\end{itemize}
Then for any $\varepsilon$ small enough there exist a total number of levels $L$ and a number of samples $N_\ell$, $\ell=0,\ldots,L$, such that the root mean square error $e(\widehat{Q}^{\ML}_h)$ is bounded by $\varepsilon$ and the total cost behaves like
\begin{equation}
\mathcal{C}\left[\widehat{Q}_h^{\ML}\right] \leq \mathcal{O}\left(
\left\{
\begin{aligned}
&\varepsilon^{-2}, & & \beta > \gamma \\
&\varepsilon^{-2}(\log \varepsilon)^2, & & \beta = \gamma \\
&\varepsilon^{-2-\frac{\gamma-\beta}{\alpha}}, & & \beta < \gamma \\
\end{aligned}
\right.
\right).
\label{eq:MLMC_totalcost_thm}
\end{equation}
\label{thm:totalcost}
\end{theorem}

Since the proof of the above theorem closely follows along the lines of \cite[Appendix A]{CliffeGiles11},
it is omitted here. The main difference results from assuming that $\mathbb{E}[Q]$ is an element of a generic Hilbert space and from the corresponding definitions of the estimator's total variance and numerical bias according to \eqref{eq:Var_Def} and \eqref{eq:MSE_splitform}.

To interpret the above theorem, it is useful to have a look at the core idea of its proof.
Starting with the basic definition of the costs \eqref{eq:MLMC_totalcost_def}, we insert the optimal choice for the number of samples \eqref{eq:optimal_Nl} and the assumptions on $V_\ell$ and $C_\ell$,
\begin{equation*}
\mathcal{C}\left[\widehat{Q}^{\ML}_h\right] = 
\sum\limits_{l=0}^{L} N_\ell C_\ell \simeq 
\frac{2}{\epsilon^2} \left(\sum\limits_{\ell=0}^{L} \sqrt{V_{\ell} C_{\ell} }\right)^2 
\simeq 
\
\varepsilon^{-2} \left(\sum\limits_{\ell=0}^{L} {H_\ell^{\frac{\beta-\gamma}{2}} }\right)^2.
\end{equation*}
A case-by-case analysis of the sum then leads to the estimate in (\ref{eq:MLMC_totalcost_thm}).
Note that the case $\beta > \gamma$ means that the total cost is dominated by the coarsest levels, whereas $\beta < \gamma$ corresponds to a case where most of the computational effort is found on the finest levels.


\section{MLMC and local time-stepping for the wave equation}
Here, we estimate the computational cost of the MLMC algorithm in the
presence of locally refined meshes, first with a standard explicit time-stepping method and then with a local time-stepping method. In doing so, we differentiate between local refinement inside a small fixed region or towards a reentrant corner. 
\subsection{Standard discretizations on a (quasi-)uniform mesh}
In \cref{alg:MLMCGiles}, the numerical method to compute any approximation for a fixed sample $\omega\in\Omega$ was not specified further.
In this section, we will take a closer look on the methods used to compute numerical solutions to (\ref{eq:Wave_hd}) for fixed $\omega\in\Omega$.
We start by discretizing the wave equation (\ref{eq:Wave_hd}) in space with either standard continuous ($H^1$-conforming) finite elements with mass-lumping or an appropriate discontinuous Galerkin (DG) discretization, for example symmetric IP-DG \cite{GSS06} or HDG \cite{ExplicitHDG16}. 
The standard continuous Galerkin formulation of the wave equation (\ref{eq:Wave_hd}) starts from its weak formulation \cite{LionsMagenes}. We then wish to approximate the solution $u(t,\cdot)$ in a suitable Finite Element space $V_H$ and thus consider the semidiscrete Galerkin approximation: find $u_H : [0,T] \to V_H$ such that
\begin{align*}
\left( \frac{\partial^2}{\partial t^2} u_H , v \right) + \left( c \nabla u_H , c \nabla v \right) &= \left( f,v \right) & &\forall\, v \in V_H, \, t \in (0,T], \\
u_H |_{t=0} &= \Pi_H u_0, \\
\frac{\partial}{\partial t} u_H |_{t=0} &= \Pi_H v_0,
\end{align*}
where $(\cdot,\cdot)$ and $\Pi_H$ denote the standard $L^2$ scalar product and the $L^2$-projection onto $V_H$, respectively.
Furthermore, let $\mathbf{u}$ denote the vector of coefficients of $u_H$ with respect to a basis $(\varphi_i)_{i=1,\ldots,n}$ of $V_H$ and let the mass and stiffness matrices, $\mathbf{M}$ and $\mathbf{K}$, together with the load vector ${\bf F}$ be defined as
\[
M_{i,j} = \left( \varphi_j , \varphi_i \right), \quad K_{i,j} = \left( c \nabla \varphi_j , c \nabla \varphi_i \right), \quad F_i(t) = \left(f(t,\cdot),\varphi_i\right),
\]
respectively. 
Here to compute the entries in the mass matrix $\mathbf{M}$, we use judicious quadrature rules which yield  a diagonal matrix while preserving the order of accuracy  (aka ''mass-lumping'') \cite{COHEN, mass_lumping_2d}.
This leads to a second-order system of ordinary differential equations
\begin{equation}
\left\{
\begin{aligned}
\mathbf{M} \ddot{\mathbf{u}}(t) + \mathbf{K u}(t) &= \mathbf{F}(t), \\
\mathbf{u}(0) &= \mathbf{u}_0, \\
\dot{\mathbf{u}}(0) &= \mathbf{v}_0.
\end{aligned}
\right.
\label{eq:WaveND_Galerkin}
\end{equation} 
Since $\mathbf{K}$ depends continuously on $c$, a random field with values in $L^\infty(D)$,
$\mathbf{u}$ also depends continuously on $c$ and is therefore measurable with respect to $\Omega$. 
Setting
$\mathbf{z} = \mathbf{M}^{1/2} \mathbf{u}$, $\mathbf{A} = \mathbf{M}^{-1/2} \mathbf{K M}^{-1/2}$ and $\widetilde{\mathbf{F}} = \mathbf{M}^{-1/2} \mathbf{F}$, we rewrite (\ref{eq:WaveND_Galerkin}) as
\begin{equation}
\ddot{\mathbf{z}}(t) + \mathbf{A z}(t) = \widetilde{\mathbf{F}}(t),
\label{eq:Wave_Galerkin_massI}
\end{equation}
which can now be discretized in time by a standard explicit time-stepping scheme. 

For $\widetilde{\mathbf{F}}_n = \widetilde{\mathbf{F}}(t_n)$, the second-order leapfrog (LF) scheme with time step $\Delta t > 0$ is then given by
\begin{equation}
\left\{
\begin{aligned}
\mathbf{z}_{n+1} - 2 \mathbf{z}_n + \mathbf{z}_{n-1} 
&= \Delta t^2 \left( \widetilde{\mathbf{F}}_n - \mathbf{A z}_n  \right) & & \forall n \geq 1, \\
\mathbf{z}_0 &= \mathbf{M}^{1/2} \mathbf{u}_0, \\
\mathbf{z}_1 &= \mathbf{z}_0 - \Delta t \,\mathbf{M}^{1/2} \mathbf{v}_0 + \frac{\Delta t^2}{2} \left( \widetilde{\mathbf{F}}_0 - \mathbf{A z}_0 \right),
\end{aligned}
\right.
\label{eq:LF2scheme}
\end{equation}
where $\mathbf{z}_n \simeq \mathbf{z}(t_n,\omega)$ for a fixed $\omega \in \Omega$. 

We now apply \cref{thm:totalcost} on the computational complexity of MLMC methods to the above discrete Galerkin formulation of the stochastic wave equation.
Although in \eqref{eq:LF2scheme} we 
opt for the popular second-order leapfrog scheme \cite{HLW}, all the estimates derived below in fact remain
identical for any standard explicit time-stepping method, such as Runge-Kutta or Adams-Bashforth methods.  

\begin{corollary}

Let $\widehat{Q}_h^{\ML}$ be the MLMC estimator of $\mathbb{E}[Q] \in V$, where $Q: C^0\left(0,T;L^2\left(D\right)\right) \to V$ is Lipschitz continuous, $D \subset \mathbb{R}^d$.
Assume $u(\cdot,\cdot,\omega)$ to be sufficiently regular uniformly in $\omega\in\Omega$ and let $Q_\ell = Q\left(u_{H_\ell}\right)$, where $u_{H_\ell}$ is computed using a finite element space discretization of order $k$ and explicit time integration of order $m$, 
such that
\begin{equation}
\left\| \mathbb{E}\left[Q_{\ell} - Q\right] \right\|_V \leq \mathcal{O}\left(\left(H_\ell\right)^{k+1} + \left(\Delta t_\ell\right)^m\right),
\label{eq:MLMCLF_WeakConvAssum}
\end{equation}
where $\Delta t_\ell$ satisfies a CFL stability condition $\Delta t_\ell \simeq H_\ell / k^2$.
Furthermore let $\beta > 0$ be a constant such that 
\(
V_\ell \leq \mathcal{O}\left(\left(H_\ell\right)^\beta\right)
\)
and
\[
2\min\{k+1,m\} \geq \min\{\beta, d+1\},
\]
and assume that the cost for evaluating the random field $c$ at the quadrature nodes is bounded by $\mathcal{O}\left(k^{2d} H_\ell^{-d}\right)$. 

Then for $\varepsilon>0$ sufficiently small, there exist a total number of levels $L$ and a number of samples $N_\ell$, $\ell=0,\ldots,L$, such that the root mean square error $e(\widehat{Q}^{\ML}_h)$ is bounded by $\varepsilon$ and the total cost behaves like
\begin{equation}
\mathcal{C}\left[\widehat{Q}_h^{\ML}\right] \leq \mathcal{O}\left(
\left\{
\begin{aligned}
&\varepsilon^{-2}, & & \beta > d+1 \\
&\varepsilon^{-2}(\log \varepsilon)^2, & & \beta = d+1 \\
&\varepsilon^{-2-\frac{d+1-\beta}{\min\{m,k+1\}}}, & & \beta < d+1 \\
\end{aligned}
\right.
\right).
\label{eq:MLMCLF_totalcost_thm}
\end{equation}
\label{cor:totalcostLF}
\end{corollary}

\begin{remark}
For $Q$ Lipschitz continuous, it is reasonable to assume for sufficiently regular solutions that the weak convergence rates assumed in \eqref{eq:MLMCLF_WeakConvAssum} are in fact a consequence of the other assumptions.
The above result can be generalized to 
rougher solutions $u(\cdot,\cdot,\omega)$ or
a broader class of quantities of interest $Q$, with possibly modified convergence rates in (\ref{eq:MLMCLF_totalcost_thm}).
\end{remark}

\begin{proof}
The computational cost $C_\ell$ of solving one wave equation with an explicit time-stepping scheme on a (quasi-)uniform mesh of size $H_\ell$ is computed as the number of time-steps times the costs in each step, which are dominated by one or more matrix-vector multiplications of type ``$\mathbf{A u}_n$''.
The cost of each such matrix-vector multiplication is approximately the number of degrees of freedom per element squared, $k^{2d}$, times the number of elements, which is proportional to $ H_\ell^{-d}$. 
Due to the CFL stability condition, the number of time-steps used on each level is inversely proportional to $\Delta t_\ell \simeq H_\ell / k^2$.
Hence, the total computational cost for solving one wave equation on level $\ell$ is
$C_\ell = c_3 \, k^{2(d+1)} \, H_\ell^{-(d+1)}$,
as by assumption the cost for evaluating the random field at the quadrature nodes is negligible.
Since $\alpha = \min\{k+1,m\}$ because of the CFL restriction on $\Delta t_\ell$, the corollary directly follows from \cref{thm:totalcost}.
\end{proof}

\subsection{Effect of local mesh refinement}
\label{sec:localref}
Due to geometric constraints or accuracy requirements, 
it may not be optimal or even possible to coarsen the entire mesh uniformly
in the presence of singularities or complex geometry, see
\cref{fig:localmesh_2Dchannel}.
\begin{figure}
\centering
\includegraphics[angle=90,width=0.7\textwidth]{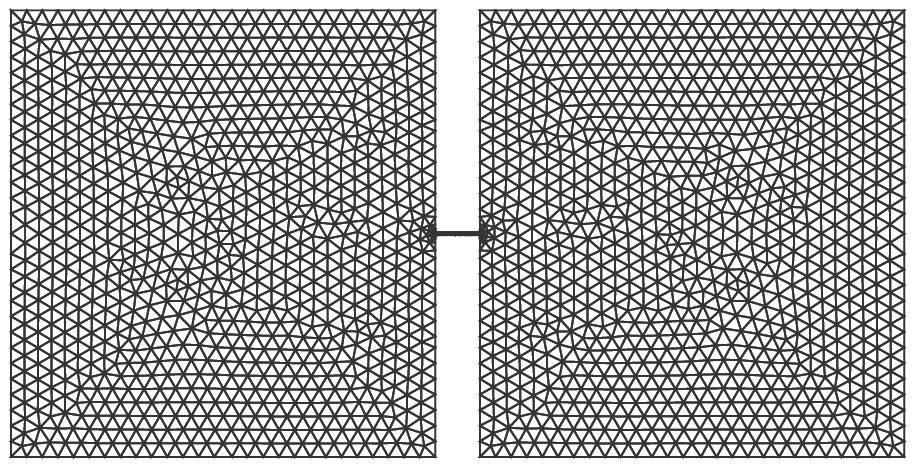}
\caption{2D Example of a computational domain with a locally refined mesh.}
\label{fig:localmesh_2Dchannel}
\end{figure}  
This results in a splitting of the computational domain, where small parts of the geometry $D_{\fine}$ might not allow for elements larger than a given $h^{\operatorname*{f}}$, and a coarse part $D_{\coar}$, where the mesh size on the coarsest level $H_0 \gg h^{\operatorname*{f}}$ can 
be much larger, while still resolving the dominant wave-length in the problem.
An example is illustrated in \cref{fig:localmesh_1D} for the domain $D = (0,6) \subset \RR$. 
\begin{figure}
\centering
\includegraphics[width=0.7\textwidth]{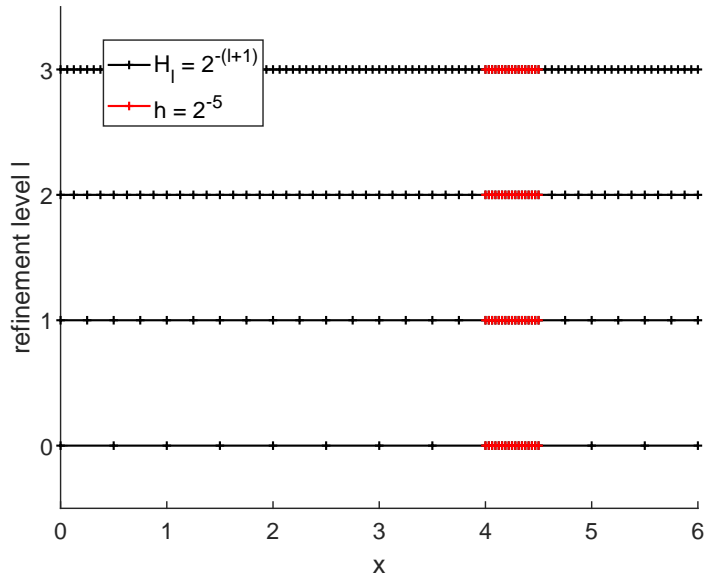}
\caption{Locally refined sequence of meshes on $[0,6]$ for levels $\ell = 0,1,2,3$. In the coarse part the mesh width is halved on each level, $h^c_\ell = 1/2^{\ell+1}$. In the fine part the width $h^f$ stays constant at $1/32$. }
\label{fig:localmesh_1D}
\end{figure}  

One of the key ideas of the MLMC method is to evaluate most of the samples on the coarsest levels and only a few on the finest. 
For explicit time-stepping schemes the maximal time-step depends on the size of the smallest elements of the mesh. 
Hence, if parts of the mesh consist of a few tiny elements across all levels,
 for every sampling the maximal time-step will be constrained by those elements in $D_{\fine}$. 
On coarser levels, standard explicit time-stepping schemes then become increasingly inefficient due to the ever more restrictive CFL condition. 

In order to quantify this effect, we want to estimate the computational cost of multilevel Monte Carlo with \FE s and leapfrog for the wave equation in the presence of local refinement on a fixed hierarchy of discretizations.
Note that another type of local mesh refinement will be 
addressed
in section \ref{sec:Lshape}, which are graded meshes for domains with a reentrant corner.

For the sake of simplicity, we consider (\ref{eq:WaveND_Galerkin}) with right-hand side equal to zero and assume that the (integer) coarse-to-fine mesh size ratio $p_\ell = \left\lceil H_\ell / h^{\operatorname*{f}} \right\rceil$ on any level $\ell=0,1,\ldots,L$ is of the form $p_\ell = p_0 / 2^\ell = 2^{L-\ell} p_L$, where we assume that $p_L \geq 1$ and,
thus, $H_L \geq h_{\min}$, i.e. the mesh in $D_{\coar}$ with mesh size $H_\ell = H_0 / 2^\ell$ remains coarser than the mesh in $D_{\fine}$ through all levels (except possibly the finest).
In the following, let $\Delta t_\ell \simeq H_\ell / k^2$ be the ``optimal'' time-step on a uniform mesh, where $k$ denotes the polynomial degree of the FEM basis functions. 
Then, the number of FEM degrees of freedom in the coarse and fine parts are respectively given by
\begin{equation}
n_\ell^{\operatorname*{c}} \simeq (1-r)k^d \frac{1}{\left(H_\ell\right)^{d}} , \qquad n^{\operatorname*{f}} \simeq r k^d \frac{1}{\left(h^{\operatorname*{f}}\right)^d} = r k^d \left(\frac{p_\ell}{H_{\ell}}\right)^d \quad\forall\ell,
\label{eq:ndofs}
\end{equation}
where $r$ denotes the relative volume of the locally refined part $D_{\fine}$ with respect to the whole domain $D$.
As the computational cost of the leapfrog method (\ref{eq:LF2scheme}) is dominated by matrix-vector multiplications $\mathbf{A  u}_n$, we will focus on these. 
With (\ref{eq:ndofs}), each of these require approximately
\[
k^d \left( n_\ell^{\operatorname*{c}} + n^{\operatorname*{f}} \right) \simeq \frac{k^{2 d}}{H_\ell^d} \left((1-r)+r p_\ell^d\right)
\]
operations. 
Due to CFL restriction, in order to proceed the simulation from $0$ to $T$, standard LF needs approximately
\[
\frac{T}{\Delta t_\ell / p_\ell} \simeq T \frac{k^2 p_\ell}{H_\ell}
\]
time-steps of size $\Delta t_\ell / p_\ell \simeq h_{\min} / k^2 \,\forall\ell$. 
Hence, ignoring constants, the cost of solving \emph{one} wave equation on any level $\ell$ for a particular $\omega\in\Omega$ approximately equals
\begin{equation}
\mathcal{C}_{\LF}\left[u_\ell(\omega)\right] \simeq
\left(T \frac{k^2 p_\ell}{H_\ell}\right) \cdot \left(\frac{k^{2 d}}{H_\ell^d} \left((1-r)+r p_\ell^d\right)\right) =
T k^{2(d+1)} \, \frac{p_\ell}{H_\ell^{d+1}} \left( (1-r)+r p_\ell^d \right).
\label{eq:cost_waveeq_LF}
\end{equation}
Thus, we can estimate the cost of computing \emph{one} sample $\Delta Q_{{\ell}}^{(i,\ell)}$ using standard LF, where we assume that the cost of computing any $Q_\ell^{(i,\ell)}$ is dominated by computing $u_\ell^{(i,\ell)}$. 
Again, we also assume that the cost for evaluating the random field $c$ at the quadrature nodes is bounded by $\mathcal{O}\left(k^{2d} H_\ell^{-d}\right)$ and thus negligible with respect to the overall cost for computing the numerical solution.
For $\ell = 0$, we have
\begin{equation*}
C_0^{\LF} = \mathcal{C}_{\LF}\left[u_0\left(\omega^{(i,\ell)}\right)\right] \simeq \frac{T k^{2(d+1)}}{H_0^{d+1}} \left( (1-r)\, p_0 + r\, p_0^{d+1} \right)
\end{equation*}
and for $\ell \geq 1$, where we use that $H_\ell = H_0 / 2^\ell$ and $p_\ell = 2^{-\ell} p_0$,
\begin{align}
C_\ell^{\LF} &= \mathcal{C}_{\LF}\left[u_\ell\left(\omega^{(i,\ell)}\right)\right] + \mathcal{C}_{\LF}\left[u_{\ell-1}\left(\omega^{(i,\ell)}\right)\right] \nonumber\\
&\simeq T k^{2(d+1)} \left( \frac{p_\ell}{H_\ell^{d+1}} \left( (1-r)+r p_\ell^d \right) + \frac{p_{\ell-1}}{H_{\ell-1}^{d+1}} \left( (1-r)+r p_{\ell-1}^d \right) \right) \nonumber\\
&\simeq \frac{T k^{2(d+1)}}{H_0^{d+1}} \left( \left(1-r\right) \frac{2^d+1}{2^d} 2^{d\ell} p_0 + r \cdot 2 p_0^{d+1} \right).
\label{eq:C_l_LF}
\end{align}
By using (\ref{eq:optimal_Nl}), we obtain
for the estimate of the total computational cost:
\begin{align}
\mathcal{C}_{\LF} &= \sum\limits_{\ell = 0}^{L} N_\ell C_\ell^{\LF} \simeq \dfrac{2}{\varepsilon^2} \left( \sum\limits_{\ell = 0}^{L} \sqrt{V_\ell C_\ell^{\LF}} \right)^2 \nonumber\\
&\simeq \frac{2\, T k^{2(d+1)}}{\varepsilon^2 H_0^{d+1}} \left( \sqrt{\left( (1-r)\, p_0+r\, p_0^{d+1} \right) V_0} \right. \nonumber\\
&\qquad\qquad\qquad \left. + \sum\limits_{\ell = 1}^{L} \sqrt{\left( \left(1-r\right) \frac{2^d+1}{2^d} 2^{d\ell} p_0 + r \cdot 2 p_0^{d+1} \right) V_\ell} \right)^2.
\label{eq:totalcost_MLMCLF}
\end{align}

\subsection{Local time-stepping}
\label{sec:LTS}
To overcome the bottleneck due to local mesh refinement on explicit time-stepping methods, we now consider explicit local time-stepping (LTS) methods \cite{DiazGrote09} for the numerical solution of  (\ref{eq:Wave_Galerkin_massI}). 
First, we split the vector of unknowns $\mathbf{z}$ into coarse and fine parts as
\[
\mathbf{z} = (\mathbf{I}-\mathbf{P}) \mathbf{z} + \mathbf{P z},
\]
where $\mathbf{P}$ is a diagonal matrix with all entries equal to zero or one, identifying the degrees of freedom in the refined part of the mesh $D_{\fine}$ and all elements adjacent to it.  Hence, $\mathbf{P z}$ contains those degrees of freedom associated with the locally refined part of the mesh.

The original leapfrog-based local time-stepping (LF-LTS) methods for solving second-order wave equations (\ref{eq:Wave_Galerkin_massI}) with arbitrarily high accuracy was proposed in \cite{DiazGrote09,GroteMitkova2010}.
Inside the ``coarse'' part of the mesh, it uses the standard LF method with a global time-step $\Delta t$.
Inside the ``fine'' part, however, the method loops over $p$ local LF steps of size $\Delta t / p$, where $p \geq H / h^{\operatorname*{f}}$ is a positive integer.
When combined with a mass-lumped conforming \cite{COHEN,mass_lumping_2d} or discontinuous Galerkin FE discretization \cite{GSS06} in space, the resulting method is truly explicit and inherently parallel; it was successfully applied to 3D seismic wave propagation \cite{MZKM13}. 
A multilevel version was later proposed \cite{DiazGrote15} and achieved high parallel efficiency on an HPC architecture \cite{Rietmann2017}.
%

Optimal convergence rates for the LF-LTS method from \cite{DiazGrote09} were derived for a conforming FEM discretization, albeit under a CFL condition where $\Delta t$ in fact depends on the smallest elements in the mesh \cite{GroteMehlinSauter18}. 
To prove optimal $L^2$ convergence rates under a CFL condition independent of $p$, a stabilized version of LF-LTS was devised recently in \cite{GroteMSauter20}.
The same algorithm was also proposed independently by the group of Hochbruck \cite{CarleHochbruckPC}.
For this method, we consider stabilized Chebyshev polynomials \cite{HV03}, based on Chebyshev polynomials of the first kind \cite{Rivlin}, denoted by $T_p$, and a stabilization parameter $0 \leq \nu \leq 1$.
Let us further define the constants
\[
\delta_{p,\nu} := 1 + \frac{\nu}{p^2}, \qquad \omega_{p,\nu} := 2 \, \frac{T_p^\prime\left(\delta_{p,\nu}\right)}{T_p\left(\delta_{p,\nu}\right)}
\]
and
\[
\beta_k := \frac{T_{k-1}\left( \delta_{p,\nu} \right)}{T_{k+1}\left( \delta_{p,\nu} \right)}, \quad \beta_{k+1/2} := \frac{ T_{k}\left( \delta_{p,\nu} \right) }{ T_{k+1}\left( \delta_{p,\nu} \right) } \quad \forall\,1 \leq k \leq p-1.
\]
For example, for $p=2$, $\nu=0.01$, we have $\delta_{p,\nu} = 1.005$ and
\begin{align*}
\omega_{p,\nu} &=  2 \, \frac{T_2^\prime\left(\delta_{p,\nu}\right)}{T_2\left(\delta_{p,\nu}\right)} =
 2 \,\frac{4 \delta_{p,\nu}}{2 \delta_{p,\nu}^2-1} 
 \approx 7.882, \\
\beta_1 &= \frac{T_{0}\left( \delta_{p,\nu} \right)}{T_{2}\left( \delta_{p,\nu} \right)} = \frac{1}{2 \delta_{p,\nu}^2-1} \approx 0.981, \\
\beta_{3/2} &= \frac{ T_{1}\left( \delta_{p,\nu} \right) }{ T_{2}\left( \delta_{p,\nu} \right) }
= \frac{\delta_{p,\nu}}{2 \delta_{p,\nu}^2-1} \approx 0.985 .
\end{align*}
Then, the stabilized LF-LTS algorithm to compute $\mathbf{z}_{n+1} \simeq \mathbf{z}(t_{n+1})$ for given $\mathbf{z}_n$, $\mathbf{z}_{n-1}$ for the wave equation, here with zero forcing for simplicity, is given as follows.

\begin{algorithm}
\caption{Stabilized LF-LTS}
\label{AlgStabLTS}
\begin{algorithmic}[1]
\STATE {
Set $\mathbf{q}_0^{n} := \mathbf{z}_{n}$ and \, $\mathbf{w}_{n} = \mathbf{A} \left( \mathbf{I} - \mathbf{P} \right) \mathbf{q}_0^{n}$.
}
\STATE {
Compute 
\[
\mathbf{q}_{1/p}^{n} = \mathbf{q}_0^{n} - \frac{1}{2} \left( \frac{\Delta t}{p} \right)^2 \frac{2p^2}{\omega_{p,\nu}\delta_{p,\nu}}
\left( \mathbf{w}_{n} + \mathbf{A} \mathbf{P}\, \mathbf{q}_0^{n} \right).
\]
}
\FOR {$m=1,\ldots,p-1$} 
\STATE{Compute
\begin{align*}
\mathbf{q}_{(m+1)/p}^{n} &= \left(1+\beta_k\right) \,\mathbf{q}_{m/p}^{n} - \beta_k \,\mathbf{q}_{(m-1)/p}^{n}  - \left( \frac{\Delta t}{p} \right)^2 \frac{2p^2}{\omega_{p,\nu}} \,\beta_{k+1/2} \left( \mathbf{w}_{n} + \mathbf{A} \mathbf{P}\, \mathbf{q}_{m/p}^{n} \right).
\end{align*}
}
\ENDFOR
\STATE {
Compute \,
\(
\mathbf{z}_{n+1} = - \mathbf{z}_{n-1} + 2 \, \mathbf{q}_{1}^{n}.
\)
}
\end{algorithmic}
\end{algorithm}
If the fraction of nonzero entries in $\mathbf{P}$ is small, the overall cost will be dominated by the computation of $\mathbf{w}_n$, which requires a single multiplication with $\mathbf{A}(\mathbf{I}-\mathbf{P})$ per time-step $\Delta t$. 
All further matrix-vector multiplications with $\mathbf{A P}$ only involve those unknowns inside, or immediately next to, the refined region.
Inside $D_{\coar}$, away from the coarse/fine interface, the algorithm reduces to the standard LF method with time-step $\Delta t$, regardless of $p$ or $\nu$. 
This is especially the case for $\mathbf{P}=0$, that is without any local time-stepping.
For $\nu=0$, \cref{AlgStabLTS} coincides with the original one from \cite{DiazGrote09}, since $\delta_{p,\nu} = 1$, and therefore $\beta_k = \beta_{k+1/2} = 1$ and $\omega_{p,\nu} = 2 p^2$.

In the following, we wish to repeat the computational cost analysis for MLMC combined with standard LF on locally refined meshes from the previous section, but this time for the (stabilized) LF-LTS method.
On any level $\ell$, solving a single wave equation will require approximately
\[
\frac{T}{\Delta t_\ell} \simeq \frac{T k^2}{H_\ell}
\]
time-steps of size $\Delta t_\ell$.
For each of these time-steps, the computational cost is dominated by $p_\ell \simeq H_\ell / h^{\operatorname*{f}}$ operations of the type ``$\mathbf{A P v}$'' and one operation ``$\mathbf{A}(\mathbf{I - P})\mathbf{v}$'', which only affect the $n^{\operatorname*{f}}$ or $n_\ell^{\operatorname*{c}}$ unknowns in the fine or coarse part of the domain, respectively.
While again ignoring constants, it follows from (\ref{eq:ndofs}) that the cost of solving \emph{one} wave equation on level $\ell$ for a particular $\omega\in\Omega$ with LF-LTS approximately equals
\begin{equation}
\mathcal{C}_{\LTSLF}\left[u_\ell(\omega)\right] \simeq
\frac{T k^2}{H_\ell} \cdot \left(\frac{k^{2 d}}{H_\ell^d} \left((1-r)+r p_\ell^{d+1}\right)\right) =
\frac{T k^{2 (d+1)}}{H_\ell^{d+1}} \left( (1-r)+r p_\ell^{d+1} \right).
\label{eq:cost_waveeq_LTS}
\end{equation}
As before, this allows us to estimate the cost of computing \emph{one} sample $\Delta Q_{{\ell}}^{(i,\ell)}$ using LF based LTS for any $\ell$,
where, again, we assume that the cost for generating the random field $c$ in the quadrature nodes is bounded by $\mathcal{O}\left(k^{2d} H_\ell^{-d}\right)$ and thus negligible with respect to the overall costs for computing the solution.
For $\ell = 0$, this results simply from inserting $\ell = 0$ in (\ref{eq:cost_waveeq_LTS}). For $\ell \geq 1$, we receive by similar arguments as before,
\begin{align}
C_\ell^{\LTSLF} &= \mathcal{C}_{\LTSLF}\left[u_\ell\left(\omega^{(i,\ell)}\right)\right] + \mathcal{C}_{\LTSLF}\left[u_{\ell-1}\left(\omega^{(i,\ell)}\right)\right] \nonumber\\
&\simeq \frac{T k^{2(d+1)}}{H_0^{d+1}} \left( \left(1-r\right) \frac{2^{d+1}+1}{2^{d+1}} 2^{(d+1)\ell} + r \cdot 2 p_0^{d+1} \right).
\label{eq:C_l_LTS}
\end{align}
With (\ref{eq:optimal_Nl}), this leads to a total computational cost estimate of
\begin{align}
\mathcal{C}_{\LTSLF} &= \sum\limits_{\ell = 0}^{L} N_\ell C_\ell^{\LTSLF} \simeq \frac{2}{\varepsilon^2} \left( \sum\limits_{\ell = 0}^{L} \sqrt{V_\ell C_\ell^{\LTSLF}} \right)^2, \nonumber\\
&\simeq \frac{2\, T k^{2(d+1)}}{\varepsilon^2 H_0^{d+1}} \left[ \sqrt{\left( (1-r)+r p_0^{d+1} \right) V_0} \right. \nonumber\\ 
&\qquad\qquad \left. + \sum\limits_{\ell = 1}^{L} \sqrt{\left( \left(1-r\right) \frac{2^{d+1}+1}{2^{d+1}} 2^{(d+1)\ell} + r \cdot 2 p_0^{d+1} \right) V_\ell} \right]^2.
\label{eq:totalcost_MLMCLTS}
\end{align}

\subsection{Standard vs. local time-stepping: a cost comparison}
\label{sec:Qeff}
Here, we estimate the increase in computational cost for MLMC with standard LF time-stepping over the (stabilized) LF-LTS.
For this, we compute the theoretical speed-up $\mathcal{S}$,
\begin{equation}
\mathcal{S}\left(d,r,p_0,L,\left\{ V_\ell \right\}\right) =  \frac{\mathcal{C}_{\LF}\left[\widehat{Q}^{\ML}_h\right]}{\mathcal{C}_{\LTSLF}\left[\widehat{Q}^{\ML}_h\right]},
\label{eq:Q_eff}
\end{equation}
where $\mathcal{C}_{\LF}$ denotes the total computational cost for computing the MLMC estimate to the solution of (\ref{eq:Wave_hd}) with the second-order LF method (\ref{eq:LF2scheme}) and $\mathcal{C}_{\LTSLF}$ the cost for computing it with the LF-LTS scheme (\cref{AlgStabLTS}).
In particular, we study the effects of various parameters on $\mathcal{S}$, such as 
the relative volume $r$ of the locally refined region, $D_{\fine}$, with respect to the entire computational domain $D$.
The quotient of (\ref{eq:totalcost_MLMCLTS}) over (\ref{eq:totalcost_MLMCLF}) yields the following proposition.

\begin{proposition}
The theoretical speed-up $\mathcal{S}$ in (\ref{eq:Q_eff}) of MLMC combined with LF-LTS, $\mathcal{C}_{\LTSLF}$, over MLMC with standard LF, $\mathcal{C}_{\LF}$, is given by  
\begin{align}
&\mathcal{S}\left(d,r,p_0,L,\left\{ V_\ell \right\}\right) =  \frac{\mathcal{C}_{\LF}}{\mathcal{C}_{\LTSLF}} \nonumber\\
&\simeq \left( \frac
{\sqrt{\left( (1-r)\, p_0+r\, p_0^{d+1} \right) V_0} + \sum\limits_{\ell = 1}^{L} \sqrt{V_\ell \left( r \cdot 2 p_0^{d+1} + \left(1-r\right) 2^{(d+1)\ell} \cdot \frac{2^d+1}{2^{d+\ell}} p_0 \right) }} 
{\sqrt{\left( (1-r)+r p_0^{d+1} \right) V_0} + \sum\limits_{\ell = 1}^{L} \sqrt{V_\ell \left( r \cdot 2 p_0^{d+1} + \left(1-r\right) 2^{(d+1)\ell} \cdot \frac{2^{d+1}+1}{2^{d+1}} \right) }}
\right)^2.
\label{eq:fullQ_eff}
\end{align}
\label{prop:Qeff}
\end{proposition}

\begin{remark}
\cref{prop:Qeff} also holds for LTS schemes based on other explicit methods, such as the fourth-order modified equation approach \cite{DiazGrote09}, Runge Kutta schemes \cite{GroteMehlinMitkova15} or Adams-Bashforth methods \cite{GroteMitkova13}.
Here, for simplicity, we have assumed the variances $\{V_\ell\}_{\ell = 0,1,\ldots,L}$ to be equal for both time integration methods, with or without LTS; in practice, this may not be true -- see also Section \ref{sec:Num_2d}. 
\label{rem:Vell}
\end{remark}

The speed-up $\mathcal{S}$ derived in \cref{prop:Qeff} calls for a more detailed interpretation.
In doing so, we restrict ourselves to the case where the mesh on the finest level $L$ is (quasi-)uniform. 
More precisely, we assume that $H_{L-1} / h^{\operatorname*{f}} \in (1,2)$ and refine the mesh in both $D_{\coar}$ and $D_{\fine}$ such that the resulting mesh on level $L$ is (quasi-)uniform with $p_L = 1$.
Then, the number of MLMC levels is simply given by $L = \left\lceil \log_2 p_0 \right\rceil$, so that $\mathcal{S}$ in (\ref{eq:fullQ_eff}) only depends on the four parameters $d$, $r$, $p_0$ and $\beta$, where $V_\ell = V_0 / 2^{\beta\ell}$. 
Hence, local time-stepping only occurs on the coarser levels $0,1,\ldots,L-1$.
Clearly, an even greater speed-up might result from a mesh locally refined even on the finest level $L$.

In \cref{fig:Q_eff_1d,fig:Q_eff_2d,fig:Q_eff_3d}, the speed-up $\mathcal{S}$ is shown as a function of the single parameters $r$ or $p_0$, while keeping all other parameters fixed, as in \cref{tab:Q_eff}.
We observe that the speed-up rapidly increases with decreasing relative volume of the refined part $r$, until it reaches the maximal speed-up of LTS over standard LF on the coarsest level $\ell=0$.
For fixed $r$, the maximal speed-up occurs for $10\leq p_0 \leq 30$, but decreases again for larger $p_0$.
At first glance, it might seem counterintuitive that $\mathcal{S}$ decreases for higher values of $p_0$.
However, as the ratio of degrees of freedom in the ``fine'' part over those in the ``coarse'' part further increases with $p_0$,
the cost for every local time-step relative to the overall cost of one global time-step also increases, which results in LTS being less efficient.
In fact, even for a single (deterministic) forward solve, the speed-up of LTS with $p$ local time-steps over standard time integration, given by the ratio of \eqref{eq:cost_waveeq_LF} over \eqref{eq:cost_waveeq_LTS}, is also maximal for the same range of $p$, as shown on the left of \cref{fig:Qeff_beta}.
\begin{table}
\caption{Fixed parameter values used in \cref{fig:Q_eff_1d,fig:Q_eff_2d,fig:Q_eff_3d} with $r$ the relative volume of the refined part, $p_0$ the local refinement factor on the coarsest level and $\beta$ the variance convergence rate, i.e. $V_\ell = V_0 / 2^{\beta\ell}$.}
\begin{center}
\begin{tabular}{c||c|c|c}
$d$ & $r$ & $p_0$ & $\beta$ \\
\hline
$1$ & $10^{-2}$ & $13$ & $4$ \\
$2$ & $10^{-4}$ & $19$ & $6$ \\
$3$ & $10^{-6}$ & $27$ & $8$
\end{tabular}
\end{center}
\label{tab:Q_eff}
\end{table}
\begin{figure}
\centering
\includegraphics[width=0.45\textwidth]{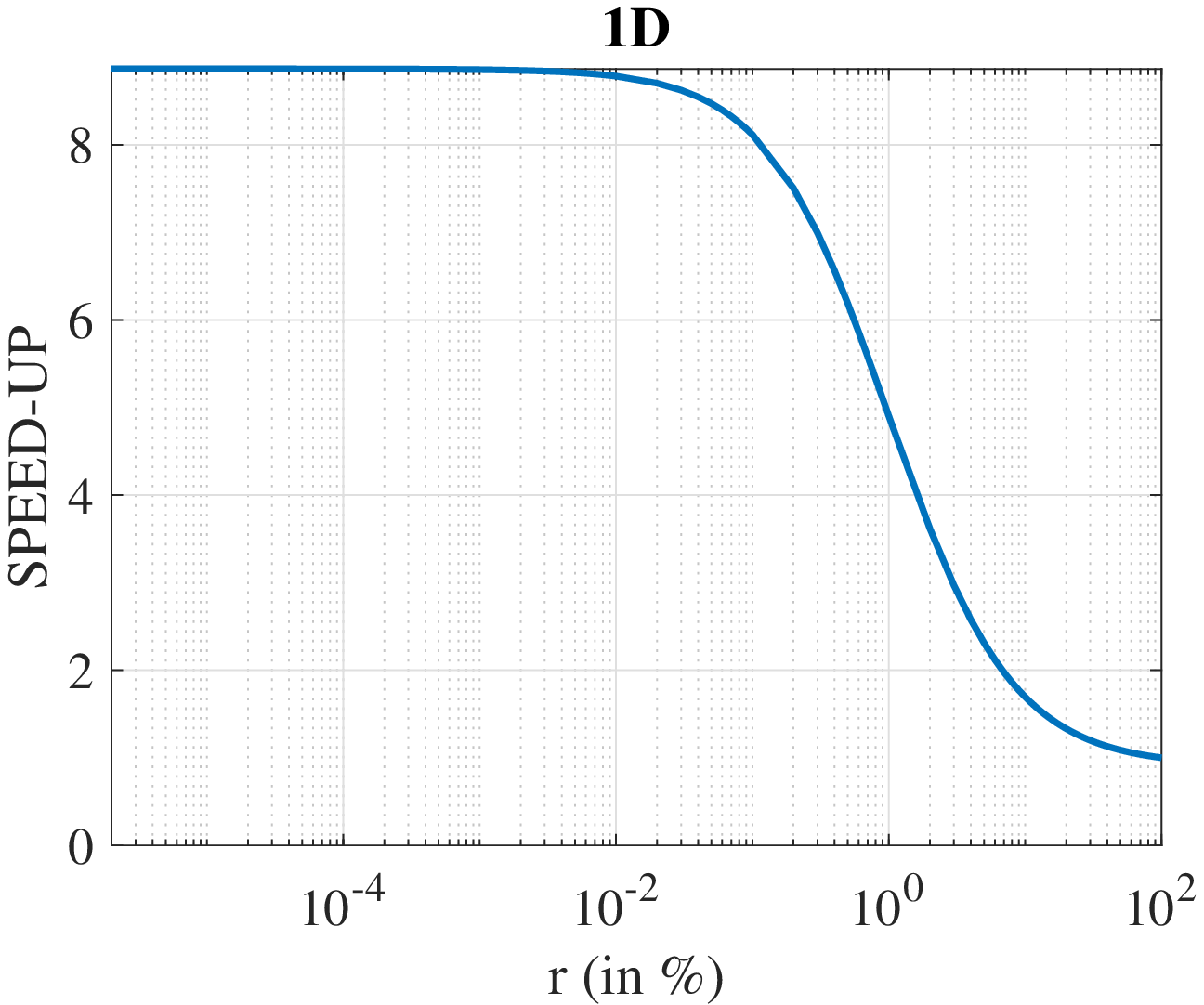}
\includegraphics[width=0.45\textwidth]{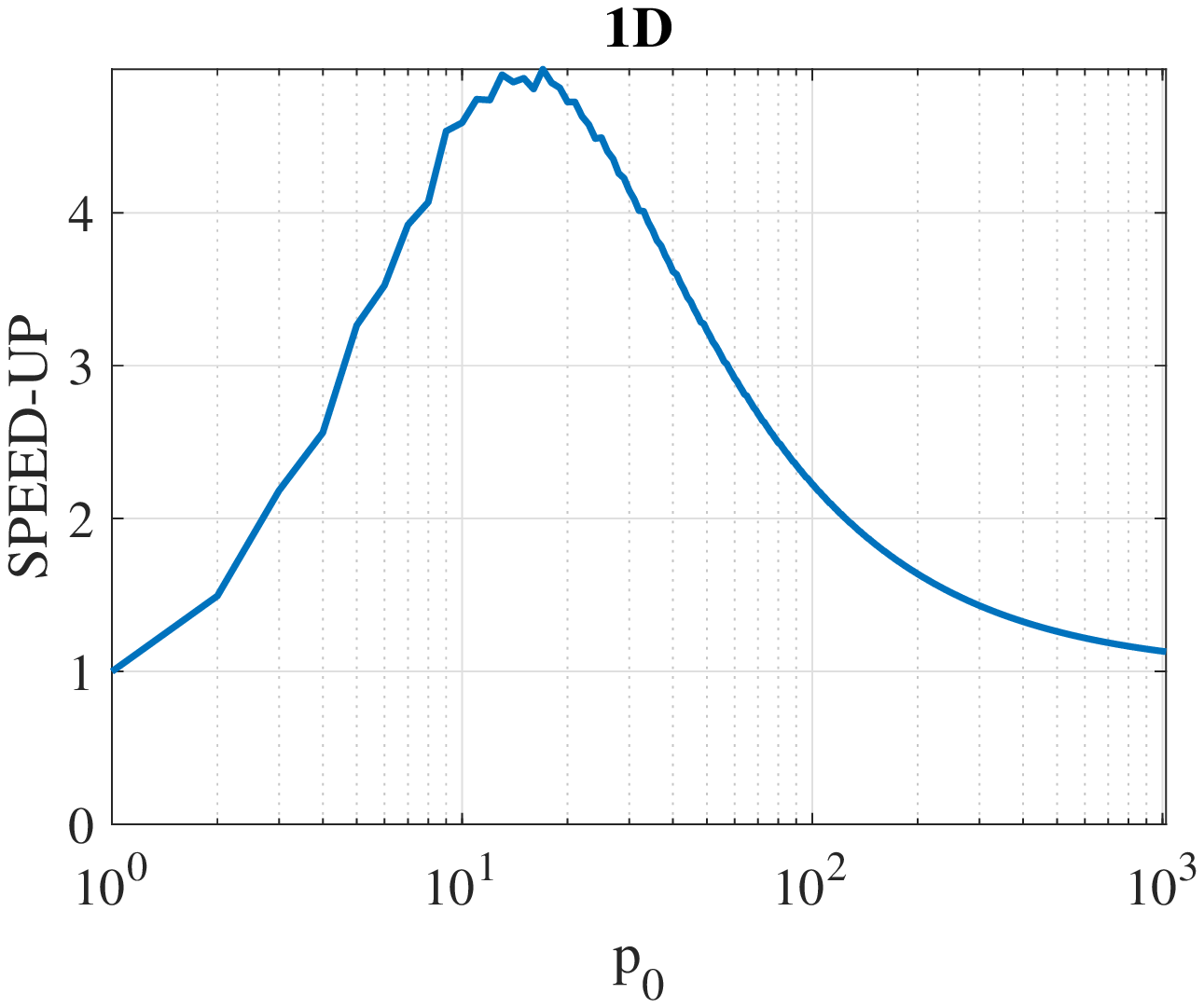} 
\caption{Theoretical speed-up $\mathcal{S}$ in 1D vs. relative volume $r$ of refined part (left) and  local refinement factor $p_0 = H_0 / h^{\operatorname*{f}}$ on the coarsest level (right) with other parameters fixed as in \cref{tab:Q_eff}.}
\label{fig:Q_eff_1d}
\end{figure}
\begin{figure}
\centering
\includegraphics[width=0.45\textwidth]{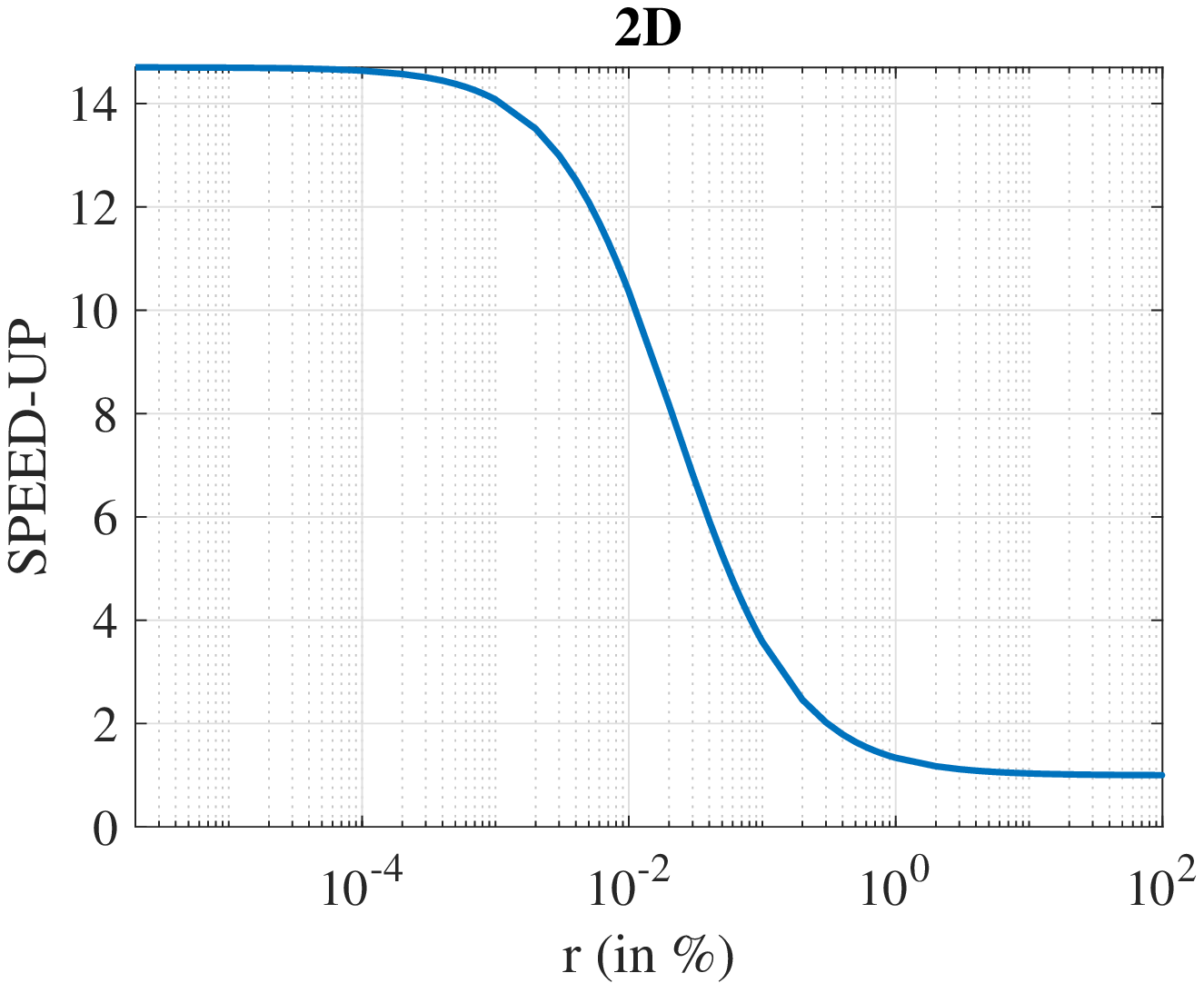}
\includegraphics[width=0.45\textwidth]{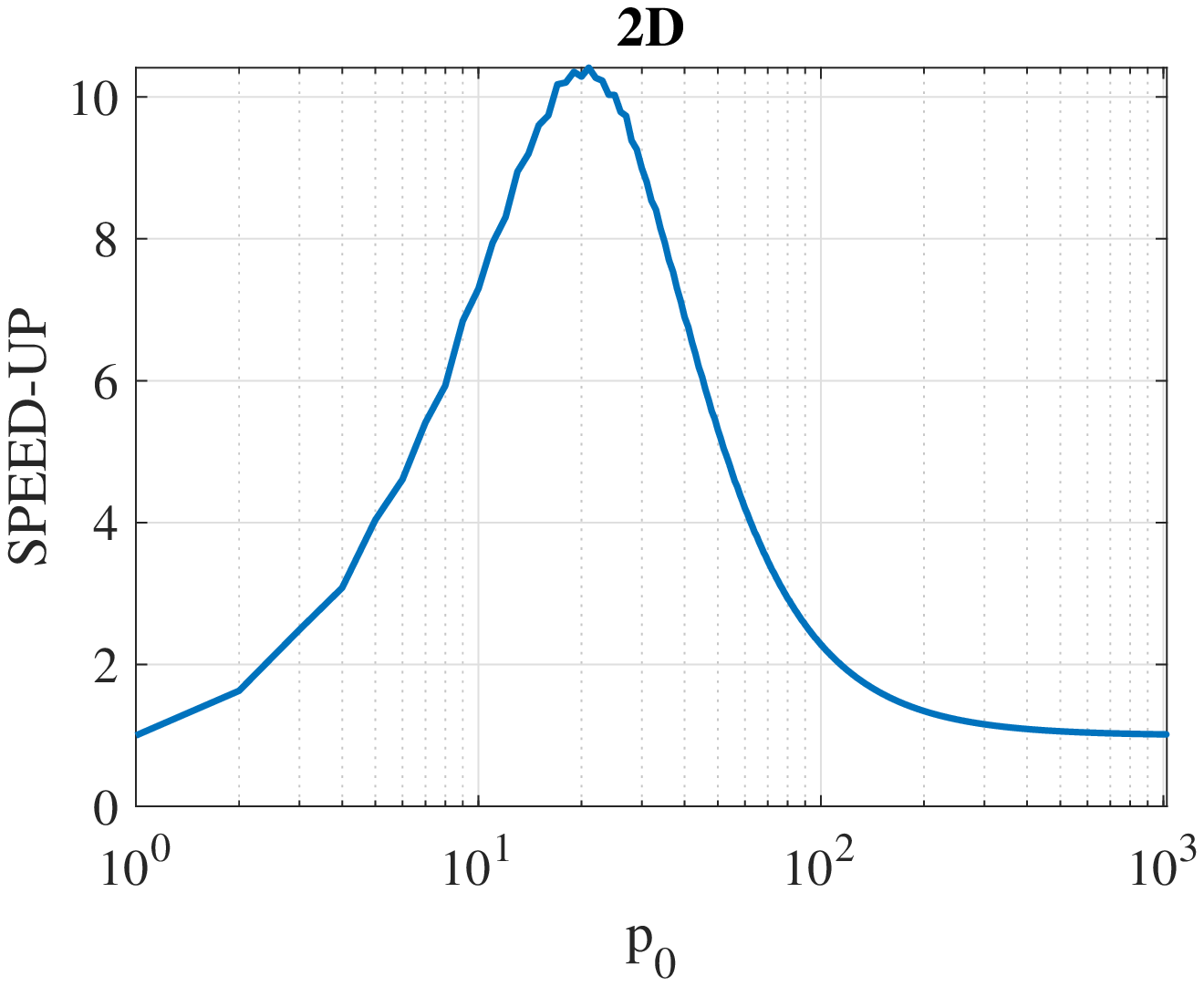} 
\caption{Theoretical speed-up $\mathcal{S}$ in 2D vs. relative volume $r$ of refined part (left) and  local refinement factor $p_0 = H_0 / h^{\operatorname*{f}}$ on the coarsest level (right) with other parameters fixed as in \cref{tab:Q_eff}.}
\label{fig:Q_eff_2d}
\end{figure}
\begin{figure}
\centering
\includegraphics[width=0.45\textwidth]{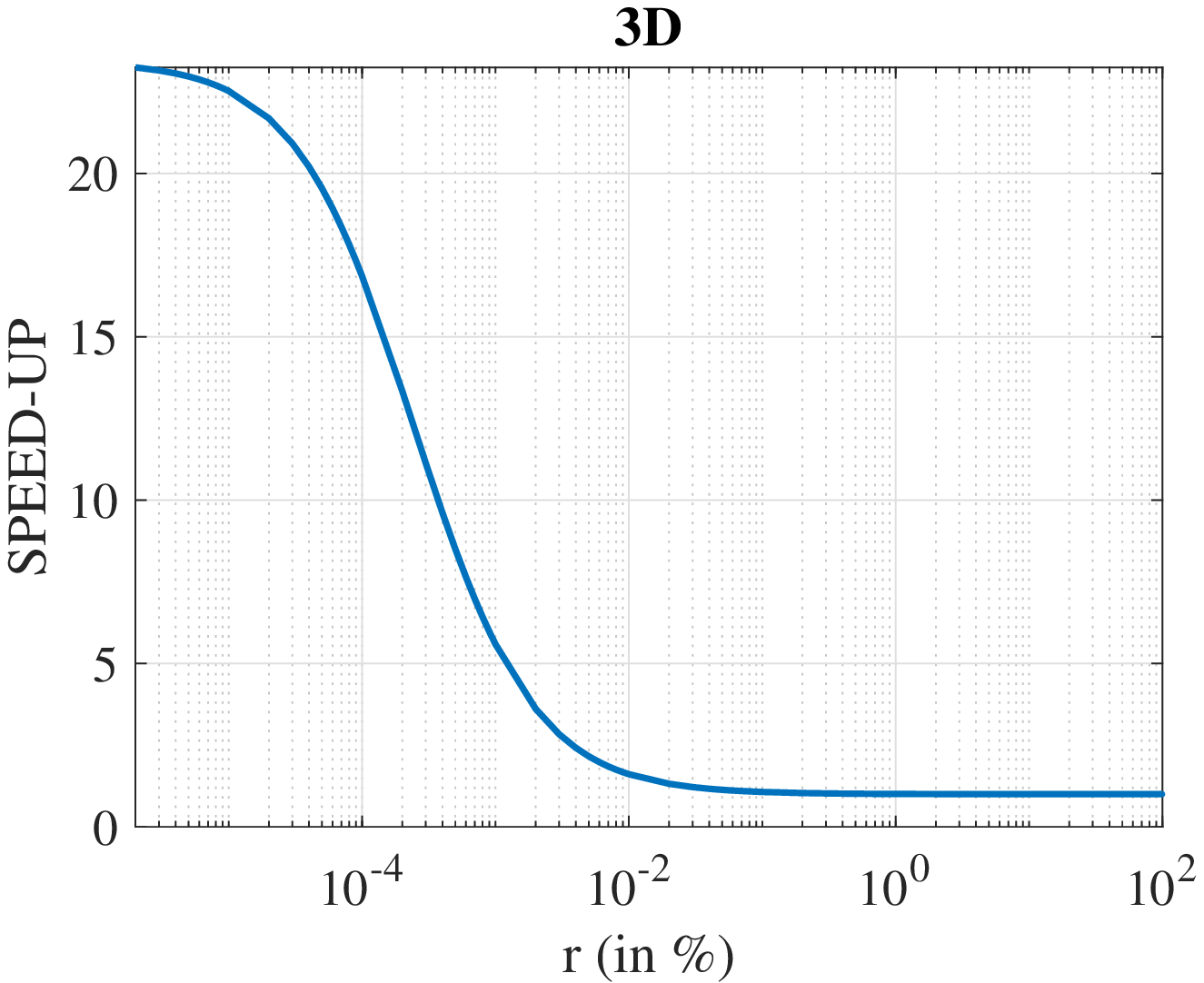}
\includegraphics[width=0.45\textwidth]{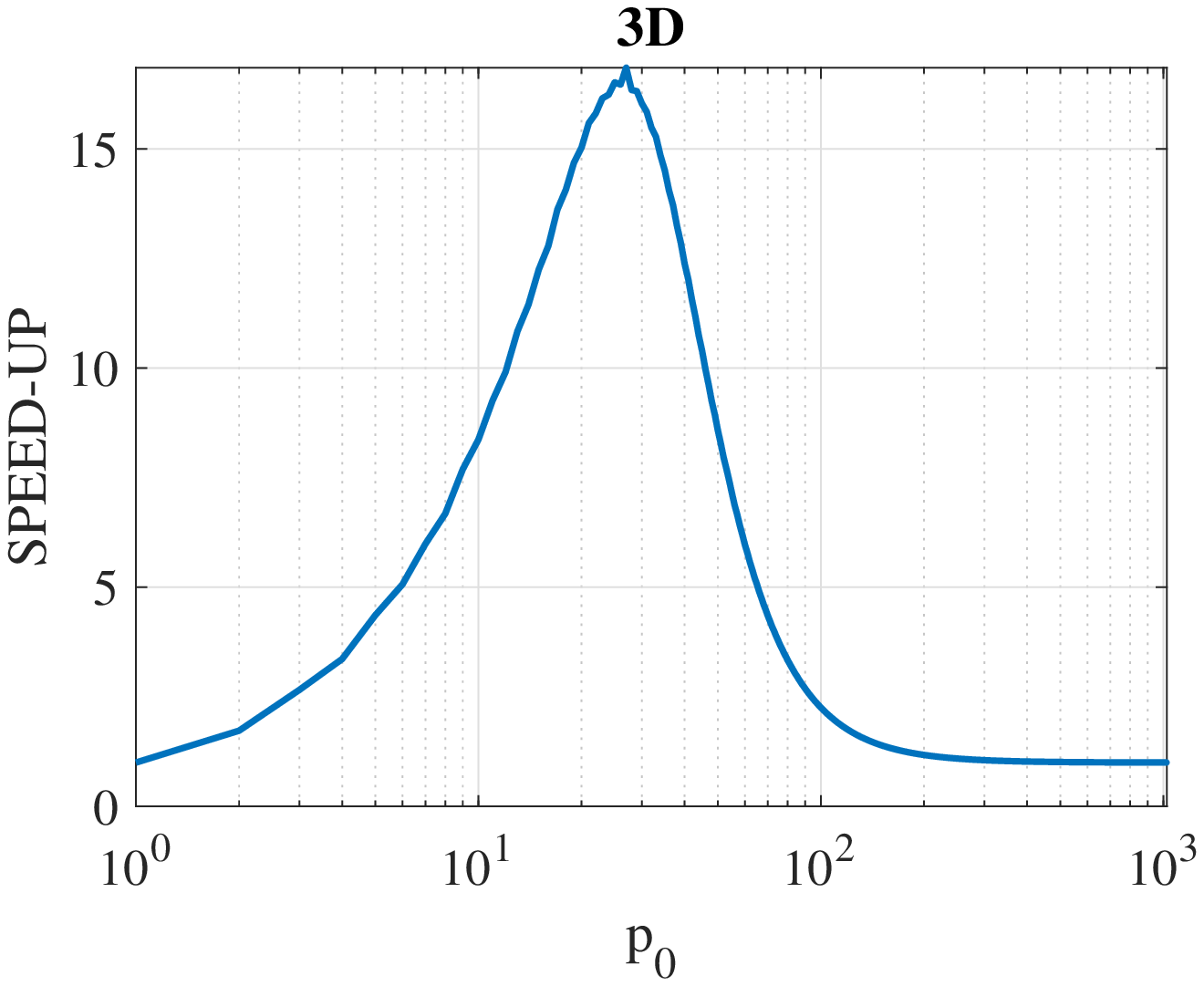} 
\caption{Theoretical speed-up $\mathcal{S}$ in 3D vs. relative volume $r$ of refined part (left) and  local refinement factor $p_0 = H_0 / h^{\operatorname*{f}}$ on the coarsest level (right) with other parameters fixed as in \cref{tab:Q_eff}.}
\label{fig:Q_eff_3d}
\end{figure}

In the right frame of \cref{fig:Qeff_beta}, we see that the performance of MLMC is improved most by LTS at higher convergence rates $\beta$ for the variance $V_\ell = V_0 / 2^{\beta\ell}$.
Indeed, the larger $\beta$, the more samples are computed on the coarsest levels, where the benefit of using LTS is the greatest.
\begin{figure}
\centering
\includegraphics[width=0.49\textwidth]{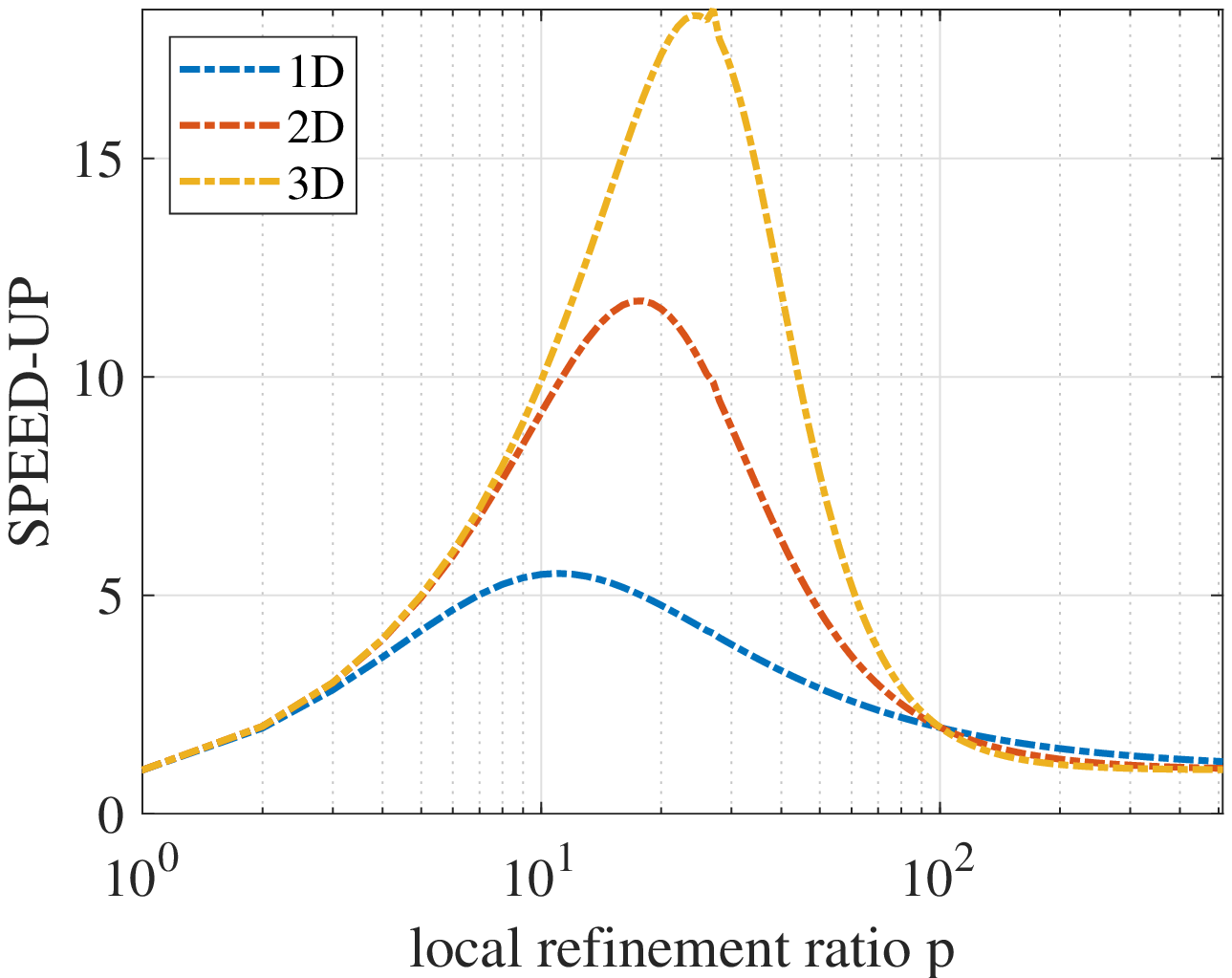}
\includegraphics[width=0.49\textwidth]{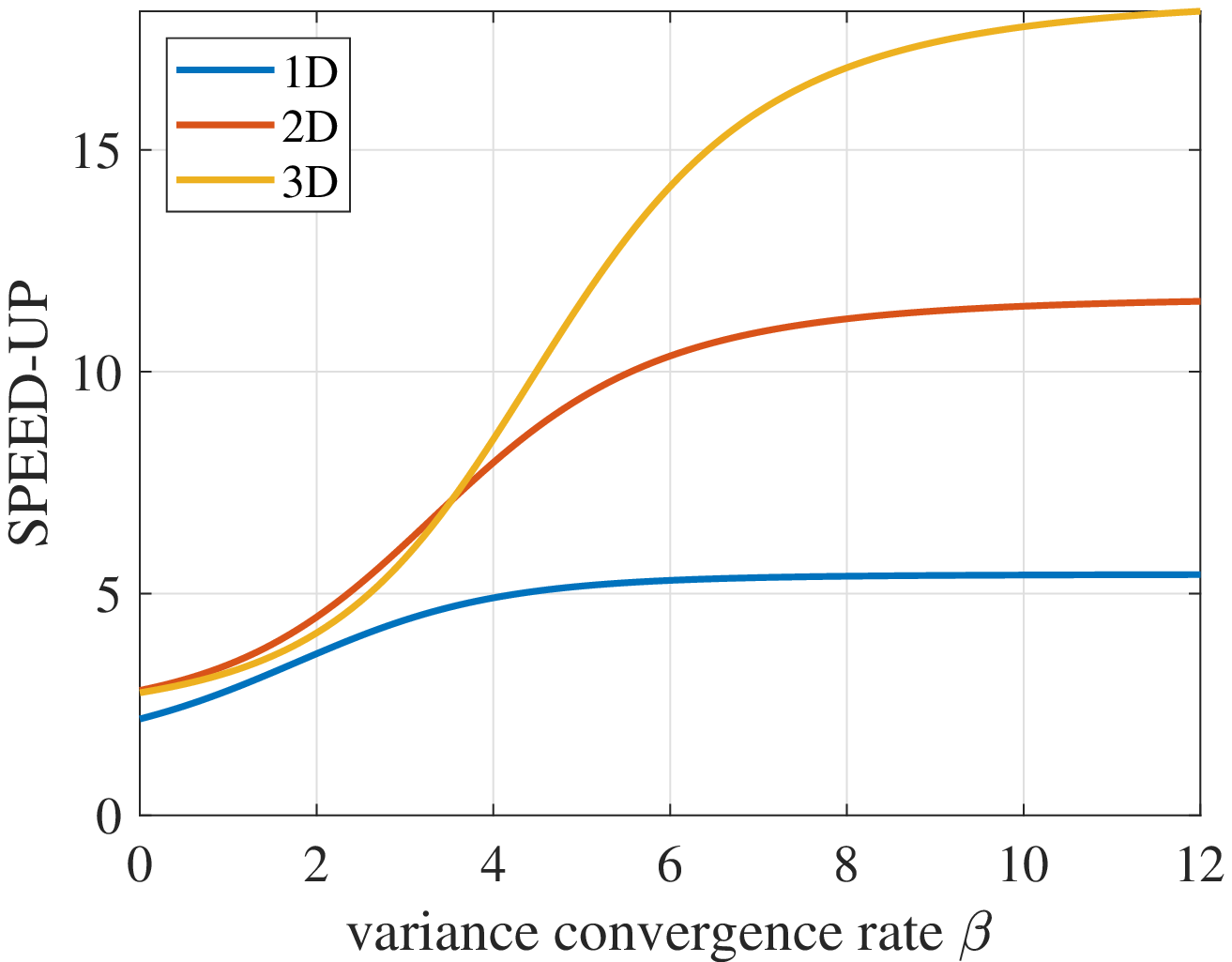}
\caption{Left: Theoretical speed-up for a single (deterministic) solve given by the ratio of \eqref{eq:cost_waveeq_LF} over \eqref{eq:cost_waveeq_LTS} relative to the number of local time-steps $p$ with relative volume of the refined region $r = 1/100^d$. 
Right: Theoretical speed-up $\mathcal{S}$ for MLMC vs. variance convergence rate $\beta$ for $V_\ell = V_0 / 2^{\beta\ell}$ with local refinement factor $p_0$ and relative volume $r$ fixed as in \cref{tab:Q_eff}.}
\label{fig:Qeff_beta}
\end{figure}

\subsection{Graded mesh refinement towards a reentrant corner}
\label{sec:Lshape}

Here, we consider a different type of local mesh refinement due to a graded mesh toward a reentrant corner.
Let $D$ be an L-shaped domain in $\mathbb{R}^2$ with a reentrant corner at $(0.5,0.5)$, shown in \cref{fig:gradedmesh}.
Due to characteristic singularities of the solution at reentrant corners, uniform meshes
generally do not yield optimal convergence rates \cite{MuellerSchwab16}. 
In the elliptic case, a common remedy to restore the accuracy and achieve optimal convergence rates is to use (a-priori) graded meshes toward the reentrant corner with appropriate weighted Sobolev spaces \cite{BabKelPit79}, see, e.g. \cite[Chapt. 3.3.7]{Schwab}.
Optimal convergence rates were also recently proved for a semi-discrete Galerkin formulation of the wave equation \cite{MuellerSchwab15} on graded meshes.

Hence, we  first partition $D$ into six triangles of equal size with a common vertex at the center $(0.5,0.5)$.
Then, on every edge $e$ connected to the center, we allocate $m+1$ points at distance 
\begin{equation*}
|e| \left(\frac{k}{m}\right)^s, \quad k = 0,1,\ldots,m,
\end{equation*}
from it, where $s \geq 1$ is a fixed grading parameter; the larger $s$, the more strongly the triangles will cluster near the reentrant corner, whereas for $s = 1$ the mesh is uniform throughout $D$.
All other vertices within the same $k$-th layer are distributed uniformly, as shown in \cref{fig:gradedmesh} for a graded mesh with $s = 2$ and $m=10$.
\begin{figure}
\centering
\includegraphics[width=0.5\textwidth]{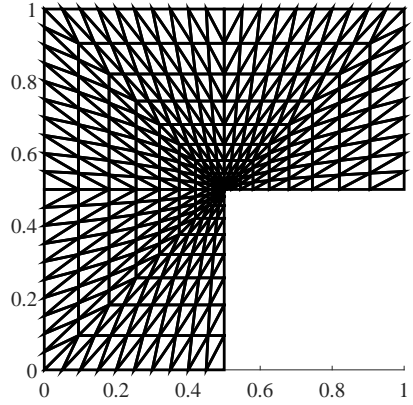}
\caption{Graded mesh on the L-shaped domain $D$ with $m=10$, $s=2$.}
\label{fig:gradedmesh}
\end{figure}
For more details on the construction of these graded meshes, we refer to \cite{MuellerSchwab15} and the references therein.
Furthermore, this strategy can be extended to general dimensions $d=1,2,3$.

By construction, the elements in a graded mesh $\mathcal{T}_{m,s}$ are distributed among $m$ different tiers or layers with equally sized triangles with
\[
h_k \approx \frac{k^s-\left(k-1\right)^s}{2m^s},\quad k=1,\ldots,m.
\]
Hence, the smallest and largest elements of the mesh are approximately of size $h_1 \approx 1/(2 m^s)$ and $h_m \approx (1-(m-1)^s/m^s)/2$, respectively.
The number of elements in the inner $k$ layers scales like $k^d$.

Thus, we can estimate the computational cost for solving the wave equation for one particular sample $\omega \in \Omega$ with any standard explicit time-stepping method by multiplying the number of time-steps with the number of elements. 
Since the time-step must be proportional to $h_1$ for stability, the cost for solving the wave equation on level $\ell$ in the MLMC algorithm with the standard LF method scales as
\begin{equation}
\mathcal{C}_{\LF}\left[u_\ell(\omega)\right] 
\simeq \frac{\mbox{no. of elements in }D}{h_1}
\simeq \frac{m_\ell^d}{h_1} \simeq m_\ell^{s+d},
\label{eq:cost_waveeq_LFgraded}
\end{equation}
where $m_\ell$ denotes the number of layers on level $\ell$, e.g. $m_\ell = m_0 \, 2^\ell$.

To estimate the computational cost for local time-stepping, we first choose the subdomain $D_{\operatorname*{f}}$, where a smaller time-step is used, as the union of the smaller first $q$ layers with $q$ still to be determined.
In $D_{\operatorname*{f}}$, the time-step is then again proportional to the smallest mesh size $h_1$, whereas in $D_{\operatorname*{c}} = 
D
\setminus D_{\operatorname*{f}}$ it is proportional to the smallest element in the outer tiers $q+1,\ldots,m_\ell$ of mesh size $h_{q+1}$.
Hence, the computational cost for solving the wave equation for one particular $\omega \in \Omega$ with an explicit LTS method scales as
\begin{align}
\mathcal{C}_{\LTSLF}\left[u_\ell(\omega)\right] 
&\simeq \frac{\mbox{no. of elements in }D_{\operatorname*{f}}}{h_1} +  \frac{\mbox{no. of elements in }D_{\operatorname*{c}}}{h_{q+1}} \nonumber\\
&\simeq q^d m_\ell^s + \left( m_\ell^d - q^d \right) \frac{m_\ell^s}{\left(q+1\right)^s-q^s} \nonumber\\
&= \frac{m_\ell^{s+d} + m_\ell^s q^d \left(\left(q+1\right)^s-q^s-1\right)}{\left(q+1\right)^s-q^s}.
\label{eq:cost_waveeq_LTSgraded}
\end{align}
The ratio of (\ref{eq:cost_waveeq_LFgraded}) to (\ref{eq:cost_waveeq_LTSgraded}) yields the expected relative speed-up
\begin{equation}
\frac{\mathcal{C}_{\LF}\left[u_\ell(\omega)\right]}{\mathcal{C}_{\LTSLF}\left[u_\ell(\omega)\right]} 
\simeq \frac{m_\ell^d \left( \left(q+1\right)^s-q^s \right)}
{m_\ell^{d} + q^d \left(\left(q+1\right)^s-q^s-1\right)}.
\label{eq:speedup_det_graded}
\end{equation}
To determine the optimal value for $q \in [1,m_\ell]$, which minimizes $\mathcal{C}_{\LTSLF}$, or equivalently maximizes the relative speed-up,
we now set the derivative of \eqref{eq:speedup_det_graded} with respect to $q$ to zero:
\begin{align}
&s \left( \left(q+1\right)^{s-1}-q^{s-1} \right) \left( m_\ell^d + q^d \left( \left(q+1\right)^s-q^s - 1 \right) \right) \nonumber\\
&- \left( \left(q+1\right)^s-q^s \right) q^{d-1} \left( d \left( \left(q+1\right)^s-q^s - 1 \right) + s q \left( \left(q+1\right)^{s-1}-q^{s-1} \right) \right) = 0. \label{eq:ddqS}
\end{align}
Since the left-hand side is positive for small $q$ and negative for larger values such as $q=m_\ell$, there exists an optimal value $q^{\operatorname*{opt}}_\ell$, $1 < q^{\operatorname*{opt}}_\ell < m_\ell$, for $m_\ell \geq 2$ and $s \geq 1$.
Since we do not seek the precise value of $q^{\operatorname*{opt}}_\ell$ but only wish to determine its asymptotic behavior as $m_\ell \to \infty$, we consider \eqref{eq:ddqS} for large $q$.
By using that, for any $\alpha > -1$,
\begin{equation}
(q+1)^\alpha - q^\alpha = \alpha q^{\alpha-1} + \mathcal{O}\left(q^{\alpha-2}\right), \quad q \to \infty,
\label{eq:qstrick}
\end{equation}
 the equation \eqref{eq:ddqS} for $q^{\operatorname*{opt}}_\ell$ reduces to
\begin{equation*}
\left( s-1 \right) m_\ell^d - s\, d\, q^{d+s-1} + \mathcal{O}\left( \frac{m_\ell^d}{q} + q^d + q^{d+s-2} \right) = 0.
\end{equation*}
For $m_\ell, q \to \infty$, the first two terms clearly dominate the other ones, which yields
\begin{equation}
q_\ell^{\operatorname*{opt}} = \mathcal{O}\left( m_\ell^{d/(d+s-1)} \right), \quad m_\ell \to \infty.
\label{eq:LTSgraded_qopt}
\end{equation}
Applying \eqref{eq:qstrick} to \eqref{eq:cost_waveeq_LTSgraded} and inserting \eqref{eq:LTSgraded_qopt} leads to
\begin{equation}
\mathcal{C}_{\LTSLF}\left[u_\ell(\omega)\right] 
\simeq \frac{m_\ell^{d+s}}{s\, q^{s-1}} + m_\ell^s q^d \left(1 - \frac{1}{s\,q^{s-1}} \right)
= \mathcal{O}\left(m_\ell^{s+d^2/(d+s-1)}\right).
\label{eq:cost_waveeq_LTSgraded2}
\end{equation}

\begin{corollary}
Let $\widehat{Q}_h^{\ML}$ be the MLMC estimator to $\mathbb{E}[Q] \in V$, where $Q$ is a Lipschitz map $C^0\left(0,T;L^2\left(D\right)\right) \to V$, $D \in \mathbb{R}^d$.
Assume $u(\cdot,\cdot,\omega)$ in a sufficiently regular weighted Sobolev space, uniformly in $\omega$, and let $Q_\ell = Q\left(u_{H_\ell}\right)$, where $u_{H_\ell}$ is computed using FE space discretization on $s$-graded meshes, $s \geq 1$, as above and explicit time integration schemes of the same order $k$, 
and
\[
\left\| \mathbb{E}\left[Q_{H_\ell} - Q\right] \right\|_{V} \leq \mathcal{O}\left(\left(H_\ell\right)^{k+1} + \left(\Delta t_\ell\right)^{k+1}\right),
\]
where $H_\ell$ is the largest element of the mesh and the time-step $\Delta t_\ell$ fulfills the CFL condition $\Delta t_\ell \simeq H_\ell / k^2$.
Further let $\beta > 0$ be a constant, such that 
\(
V_\ell \leq \mathcal{O}\left(\left(H_\ell\right)^\beta\right)
\)
and $2(k+1) \geq \min\{\beta, s+d^2/(d+s-1)\}$,
and assume that the cost for evaluating the random field $c$ at the quadrature nodes is bounded by $\mathcal{O}\left(k^{2d} H_\ell^{-d}\right)$. 

Then for any $\varepsilon$ small enough, there exist a total number of levels $L$ and a number of samples $N_\ell$, $\ell=0,\ldots,L$, such that the root mean square error $e(\widehat{Q}^{\ML}_h)$ is bounded by $\varepsilon$.

If standard time-stepping is used, the total cost for $\varepsilon \to 0$ behaves like
\begin{equation}
\mathcal{C}\left[\widehat{Q}_h^{\ML}\right] \leq \mathcal{O}\left(
\left\{
\begin{aligned}
&\varepsilon^{-2}, & & \beta > s+d \\
&\varepsilon^{-2}(\log \varepsilon)^2, & & \beta = s+d \\
&\varepsilon^{-2-\frac{s+d-\beta}{k+1}}, & & \beta < s+d \\
\end{aligned}
\right.
\right),
\label{eq:MLMCLF_totalcost_thm_graded}
\end{equation}
whereas if LTS is used, the total cost behaves like
\begin{equation}
\mathcal{C}\left[\widehat{Q}_h^{\ML}\right] \leq \mathcal{O}\left(
\left\{
\begin{aligned}
&\varepsilon^{-2}, & & \beta > s+d^2/(d+s-1) \\
&\varepsilon^{-2}(\log \varepsilon)^2, & & \beta = s+d^2/(d+s-1) \\
&\varepsilon^{-2-\frac{s+d^2/(d+s-1)-\beta}{k+1}}, & & \beta < s+d^2/(d+s-1) \\
\end{aligned}
\right.
\right).
\label{eq:MLMCLTS_totalcost_thm_graded}
\end{equation}
\label{cor:totalcostLFgraded}
\end{corollary}

\begin{proof}
The result follows from similar arguments as for \cref{cor:totalcostLF} by applying \cref{thm:totalcost} with $\gamma = s+d$ for standard time-stepping (\ref{eq:cost_waveeq_LFgraded}) and $\gamma = s + d^2/\left(d+s-1\right)$ for LTS (\ref{eq:cost_waveeq_LTSgraded2}).
\end{proof}

\begin{remark}
The ratio of \eqref{eq:MLMCLF_totalcost_thm_graded} to \eqref{eq:MLMCLTS_totalcost_thm_graded} yields  the theoretical speed-up \eqref{eq:Q_eff} for $\varepsilon \to 0$:
\begin{equation}
\mathcal{S} \leq 
\left\{
\begin{aligned}
&\mathcal{O}\left(1\right), & & \beta > s+d, \\
&\mathcal{O}\left( \varepsilon^{-\frac{s+d-\beta}{k+1}}\right), & & s+d^2/(d+s-1) < \beta < s+d, \\
&\mathcal{O}\left(\varepsilon^{-\frac{d \left(s-1\right)}{\left(k+1\right)\left(s+d-1\right)}}\right), & & \beta < s+d^2/(d+s-1). \\
\end{aligned}
\right.
\label{eq:Qeff_graded}
\end{equation}
For $\beta > s+d$, the total cost is dominated by the computational effort on the coarsest levels for both time-stepping methods; hence, the total cost has the same asymptotic behavior up to a constant factor.
In the two other cases with $\beta < s+d$, however, the speed-up of using LF-LTS over standard LF actually grows as the error tolerance $\varepsilon$ decreases. In other words, the smaller the desired error level  $\varepsilon$, the larger the gain in using local time-stepping in MLMC.
For $\mathcal{P}^1$-elements in 
$d=3$
space dimensions with grading parameter $s=2$, for instance, the speed-up grows like $\mathcal{S} = \mathcal{O}\left( \varepsilon^{-
3/8
} \right)$ if $\beta < 
17/4
$.
\end{remark}

For the above cost estimation, the computational mesh is divided into a "coarse" and a "fine"  region, each associated with a distinct time-step. Since the region of local refinement itself contains 
sub-regions of further refinement, those “very fine” elements yet again will dictate the time-step, 
albeit local, to the entire “fine” region. Then, it would be more efficient to let the time-marching strategy mimic the multilevel hierarchy of the mesh organized into tiers of “coarse”, “fine”, “very fine”, etc. elements by introducing a corresponding hierarchy into the time-stepping method. By using within each tier of equally sized elements the corresponding optimal time-step, the resulting multi-level local time-stepping (MLTS) method \cite{DiazGrote15}
would achieve an even greater speed-up.

\section{Numerical results}
\label{sec:numerics}

To illustrate the efficiency of the combined LF-LTS-MLMC algorithm, we now apply it in three distinct situations with random wave speed.

\subsection{Continuous random wave speed}
\label{sec:Num_1d_smooth}

First, we consider \eqref{eq:Wave_hd} in $D = (0,6)$ with $f \equiv 0$, $u_0(x) = \operatorname*{e}^{-(x-3)^2 / 0.09}$, $v_0(x) \equiv 0$ and homogeneous Neumann boundary conditions.
The random wave speed $c^2$ is given by the Karhunen-Lo\`{e}ve expansion,
\begin{equation*}
c^2(x,\omega) = 1 + \sum\limits_{k=1}^{100} \frac{1}{4 \pi^2 k^2} \left(\cos\left(\frac{k \pi x}{6}\right)\xi_k^{(1)}(\omega) +\sin\left(\frac{k \pi x}{6}\right)\xi_k^{(2)}(\omega)\right) ,
\end{equation*}
where $\xi_k^{(1)}, \,\xi_k^{(2)} \sim U(-1,1)$ are i.i.d. uniform random variables. 
In \cref{fig:c2real_cossin}, we display different realizations of $c^2$ for different random samples on different levels 
with $L=4$.
\begin{figure}
\centering
\includegraphics[width=\textwidth]{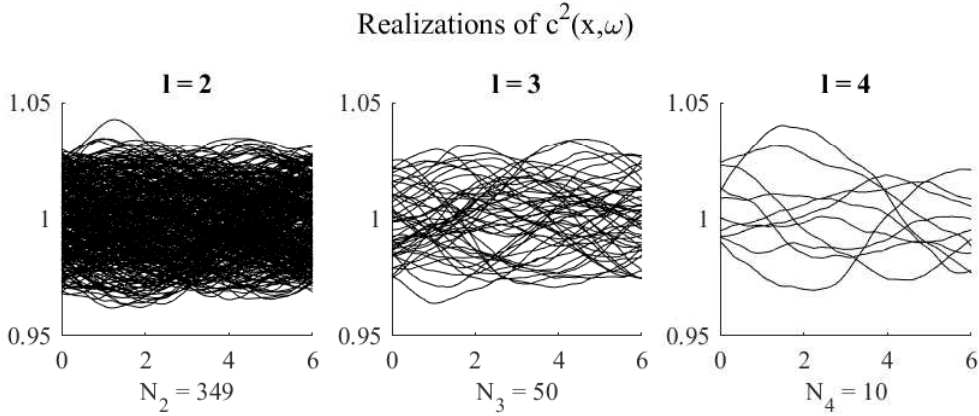}
\caption{Continuous random wave speed. Realizations of the random wave speed $c(x,\omega^{(i,\ell)})$ on levels $\ell = 2,3,4$ 
with $L=4$.
}
\label{fig:c2real_cossin}
\end{figure}

For discretization in space, we choose continuous, piecewise linear $\mathcal{P}^1$ \FE s on meshes of size $H_\ell = 2^{-(\ell+4)}$ for $\ell = 0,1,2,\ldots$.
For time discretization, we use the (stabilized) LF-LTS scheme listed in \cref{AlgStabLTS}.
Here, we arbitrarily set the locally refined part of the mesh to $D_{\fine} = [5-H_0,5]$ with elements of size $h^{\operatorname*{f}} = 2^{-8}$.
In \cref{fig:samplesol_c2cossin}, the respective FE-solu\-tions $u_{\ell}(x,T,\omega^{(i)})$ of (\ref{eq:Wave_hd}) at $T = 11$ on level $\ell = 2$ are shown for 10 particular samples $c(x,\omega^{(i)})$.
\begin{figure}
\centering
\includegraphics[width=0.65\textwidth]{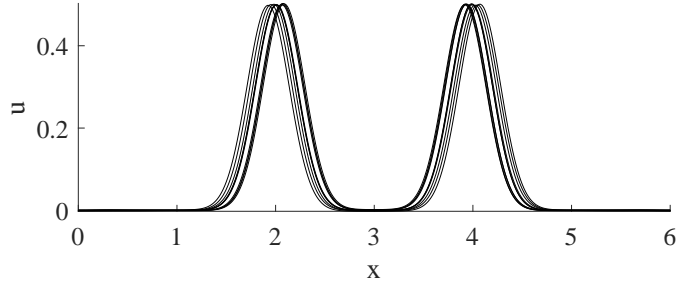}
\caption{Continuous random wave speed. Overlay of solutions $u(x,t,\omega)$ at time $T = 11$ for different random samples of smooth wave speed $c^2(\omega)$.}
\label{fig:samplesol_c2cossin}
\end{figure}

Next, we apply the MLMC \cref{alg:MLMCGiles} to estimate  the expected value of the mean solution at a fixed time $T>0$, $\mathbb{E}[u(\cdot,T,\cdot)] \in V = L^2(D)$.
In \cref{fig:worktol_cossin}, we compare the computational cost of MLMC with LF-LTS or with standard LF relative to the root mean square error $\varepsilon$ according to \eqref{eq:MSE_splitform}, where the total variance and bias term are estimated using \eqref{eq:Vell_est} and \eqref{eq:bias_est} with $\alpha=2$, respectively. 
Here, the computational cost is either measured via CPU time or estimated using
\[
\mathcal{C}\left[\widehat{Q}^{\ML}_h\right] = \sum_{\ell=0}^L N_\ell C_\ell,
\]
where $N_\ell$ 
and $L$ are
determined on the fly by the MLMC algorithm and $C_\ell$ is estimated from \eqref{eq:cost_waveeq_LF} for standard LF or \eqref{eq:cost_waveeq_LTS} for LF-LTS.
As expected from \cref{thm:totalcost}, the costs of both algorithms behave inversely proportional to $\varepsilon^2$. 
Nonetheless, the MLMC algorithm using LTS is about one order of magnitude cheaper than that using standard LF.

\begin{figure}
\centering
\includegraphics[width=0.49\textwidth]{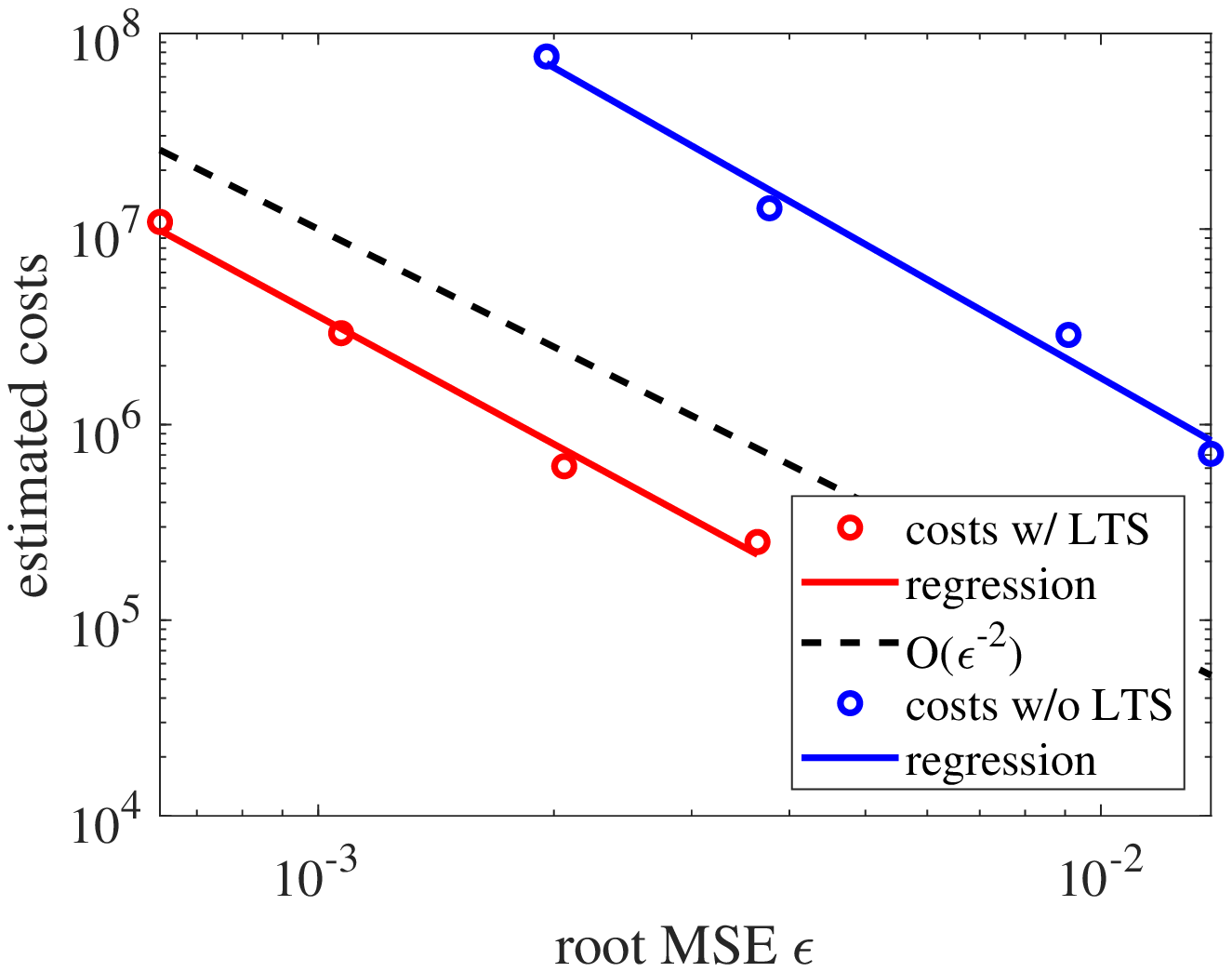}
\includegraphics[width=0.49\textwidth]{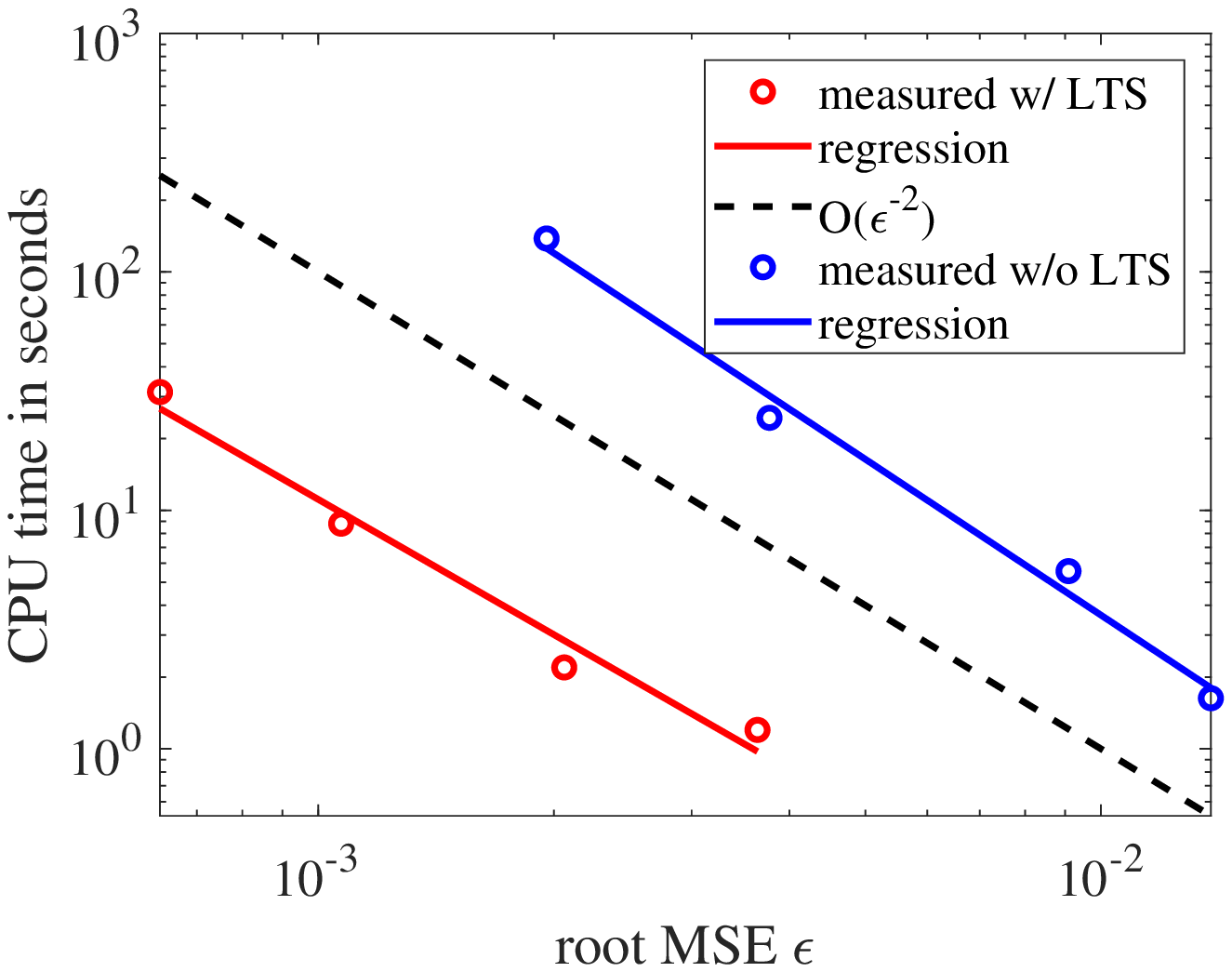}
\caption{Continuous random wave speed. Work to error ratio for the MLMC-FE leapfrog method with, or without, LTS. Left: pre-estimated costs, right: measured CPU times.}
\label{fig:worktol_cossin}
\end{figure}

\subsection{Discontinuous random wave speed}
\label{sec:Num_1d_jump}

Next, we again consider \eqref{eq:Wave_hd} as in the previous example, yet with a
discontinuous, piecewise constant wave speed,
\[
c\left(x,\omega\right) = \left\{
\begin{aligned}
&1, & &0 \leq x < \xi\left(\omega\right), \\
&2, & &\xi\left(\omega\right) \leq x \leq 6.
\end{aligned}
\right.
\]
Here, the exact jump position is not known precisely and thus modelled as a uniform random variable $\xi \sim U(4-H_0,4)$ centered about $x=4 
-H_0/2
$, with the mesh size on the coarsest level $H_0 = 1/16$.
Inside $[4-H_0,4]$, the mesh is locally refined with $h^{\operatorname*{f}} = 2^{-9}$,
independently of the realization of the random variable.
For spatial discretization, we choose continuous $\mathcal{P}^1$ \FE s.
For the time discretization, we use the leap-frog method (\ref{eq:LF2scheme}), either with or without LTS, as in \cref{AlgStabLTS}. 

Again, we apply the MLMC \cref{alg:MLMCGiles}, either with or without LTS, to estimate  the expected value of the mean solution at a fixed time $T>0$, $\mathbb{E}[u(\cdot,T,\cdot)] \in V = L^2(D)$.
On the left, \cref{fig:samplesol_c2jump} shows an overlay of six particular samples $c(x,\omega^{(i)})$ and the respective FE-solu\-tions $u_{\ell}(x,T,\omega^{(i)})$ of (\ref{eq:Wave_hd}) at $T = 6$ on the coarsest level.
On the right, we compare the performance of MLMC using either LF-LTS or the standard LF method with respect to the root mean square error $\varepsilon$ according to \eqref{eq:MSE_splitform}, where again the total variance and bias term are estimated using \eqref{eq:Vell_est} and \eqref{eq:bias_est} with $\alpha=2$, respectively. 
Although the total computational cost behaves inversely proportional to $\varepsilon^2$ in both cases, the LF-LTS based MLMC method achieves a significant reduction in overall computational cost.

\begin{figure}
\centering
\includegraphics[width=0.49\linewidth]{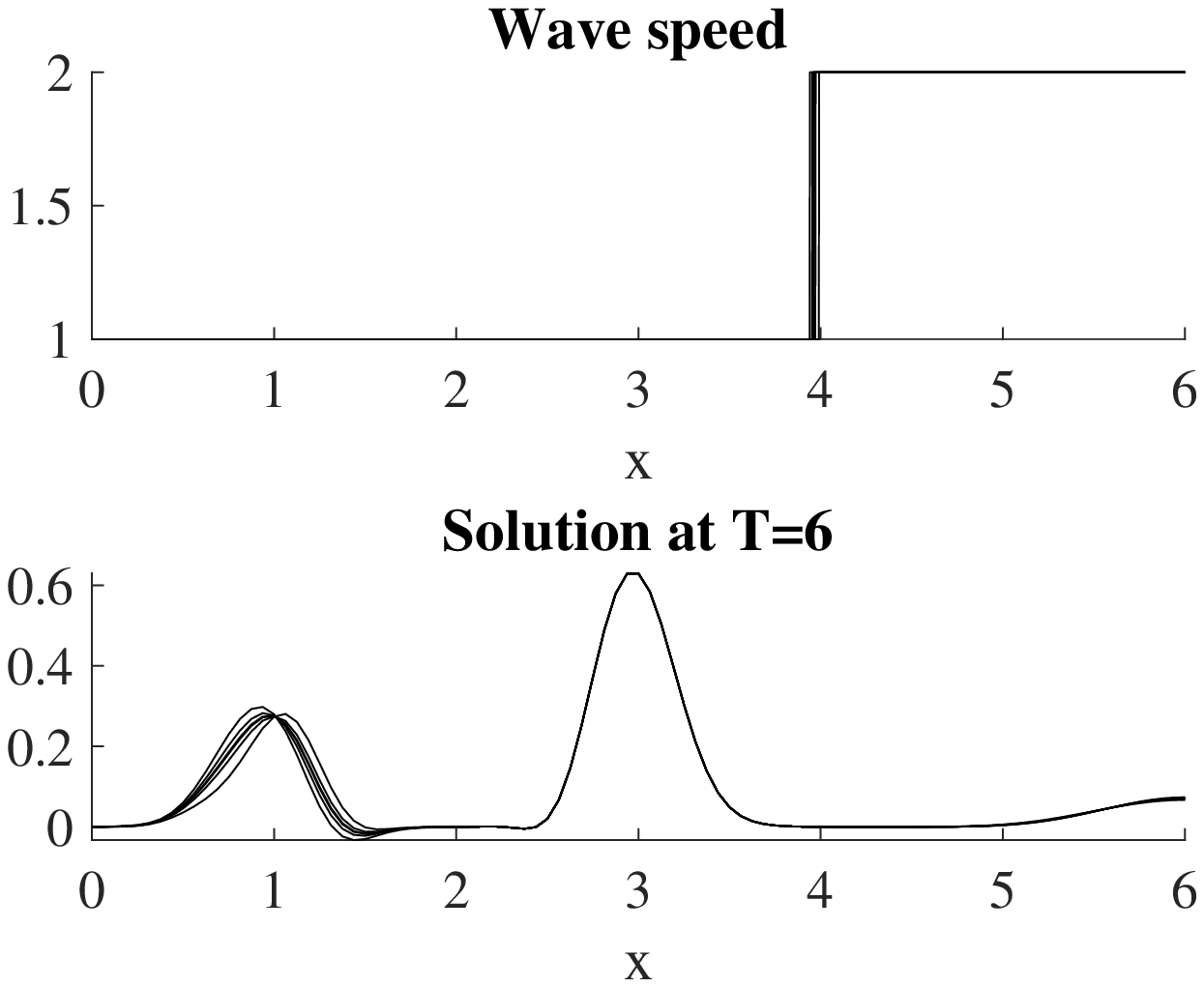}
\includegraphics[width=0.49\linewidth]{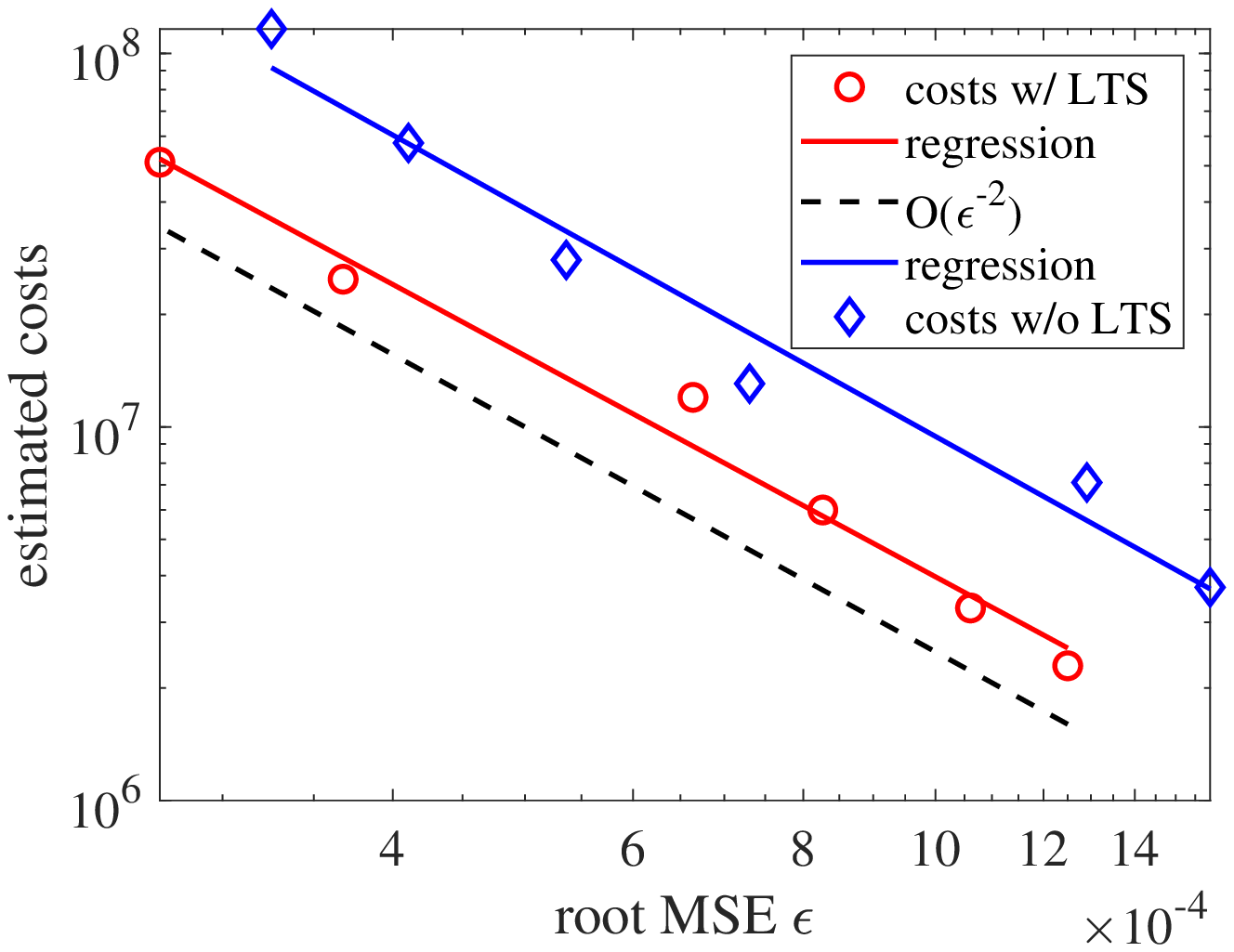}
\caption{Discontinuous random wave speed. Left: Overlay of realizations of the wave speed $c(x,\omega^{(i,0)})$ (top) and FE solutions $u_{0}(x,T,\omega^{(i,0)})$ at $T=6$ (bottom). 
Right: Computational work vs. root MSE tolerance $\varepsilon$ with, or without, LTS.}
\label{fig:samplesol_c2jump}
\end{figure}

\subsection{Two-dimensional narrow channel}
\label{sec:Num_2d}
Finally, we consider wave propagation with constant unit speed $c \equiv 1$, vanishing source $f \equiv 0$ and homogeneous Neumann boundary conditions in the two-dimensional domain $D_{b(\omega)}$, shown in \cref{fig:localmesh_2Dchannel},  with varying random width $b(\omega)$ of the narrow channel connecting the two rectangular regions.
This leads to the weak formulation: 
Find $u: [0,T] \times \Omega \to H^1\left(D_{b(\omega)}\right)$ such that
\begin{equation}
\frac{\partial^2}{\partial t^2} \left(u\left(t,\omega\right),v\right) + \left(\nabla u\left(t,\omega\right), \nabla v \right) = 0 \qquad\forall v \in H^1\left(D_{b(\omega)}\right).
\label{eq:WeakWave2d}
\end{equation}
Here, $D_{b(\omega)}$ consists of two $0.95 \times 1$ rectangles connected by a narrow channel $0.1 \times b(\omega)$ of varying random width $b(\omega)\in [0.001, 0.007]$, where the origin of the coordinate axes is located at the center of the narrow channel.
Next, we reformulate \eqref{eq:WeakWave2d} as in \eqref{eq:Wave_hd} on a (deterministic) reference domain $\bar{D}$ of fixed channel width $\bar{b} = 0.004$, yet with inhomogeneous random velocity $c(x,\omega)$.
To do so, we introduce the geometric transformation
\begin{equation}
\left\{
\begin{array}{rrcl}
\Psi(\omega) : & \bar{D} &\to& D_{b(\omega)}, \\
& (x,y) &\mapsto& \left(x , y + \varphi\left(x\right)\psi \left(y,\omega\right) \right),
\end{array}
\right.
\label{eq:MeshTransform}
\end{equation}
which maps the reference domain $\bar{D}$ to the actual computational domain $D_{b(\omega)}$.
Here, $\varphi$ is smooth and  $\psi$ piecewise linear and continuous:
\[
\varphi(x) = \left\{
\begin{array}{ll}
1, & |x| \leq 0.05,\\
0, & |x| \geq 0.1,
\end{array}
\right.
\qquad
\psi(y,\omega) = \left\{
\begin{array}{ll}
0, & y = 0, \pm 0.5, \\
\pm \left( \frac{b\left(\omega\right)}{2} - \frac{\bar{b}}{2} \right), & y = \pm \frac{\bar{b}}{2}.
\end{array}
\right.
\]
Then, \eqref{eq:WeakWave2d} is equivalent to 
\begin{equation}
\frac{\partial^2}{\partial t^2} \left( \left|J\left(\omega\right)\right| u,v \right) + \left(\left|J\left(\omega\right)\right| J\left(\omega\right)^{-\top} \nabla u , J\left(\omega\right)^{-\top} \nabla v \right) = 0 \qquad\forall v \in H^1\left(\bar{D}\right)
\label{eq:waveeq_transformed}
\end{equation}
with $J\left(\omega\right)$ the Jacobian of $\Psi(\omega)$.
As initial conditions we choose the Gaussian pulse
\begin{align*}
u_0(\mathbf{x}) &= \left\{
\begin{array}{ll}
\exp\left( 1 - \frac{R^2}{R^2 - ||\mathbf{x-x}_0||^2} \right), & ||\mathbf{x-x}_0|| < R, \\
0, &\mbox{else},
\end{array}
\right.
\\
v_0(\mathbf{x}) &= 0
\end{align*}
centered about $\mathbf{x}_0 = (0.5,0)$ and $R = 0.2$.

For the spatial discretization inside $\bar{D}$, we use continuous, piecewise linear, finite elements and generate a sequence of triangular meshes 
independent of the realization $b(\omega)$, with
mesh size $H_\ell = 1 / 60 \cdot 2^{-\ell}$, $\ell =0,1,2,\ldots$, outside the channel.
Inside the channel we use a mesh size $h^{\operatorname*{f}} \approx 7.6 \cdot 10^{-4}$ to resolve the narrow gap geometry.
For simplicity	 
in the implementation,
we simulate the varying channel width by applying the transformation \eqref{eq:MeshTransform} directly to the vertices 
of the mesh
for each realization of the uniform random variable $b \sim U(0.001,0.007)$.
For time discretization, we use the (stabilized) LF-LTS method listed in \cref{AlgStabLTS}.

We now apply the MLMC \cref{alg:MLMCGiles} to estimate as QoI the expected value of the mean solution along the vertical line $x=-0.4$ at time $T=1$, 
\[
\mathbb{E}[Q(u)](y) := \mathbb{E}\left[\left. u\left(\mathbf{x},T,\cdot\right)\right|_{\mathbf{x} = (-0.4,y)} \right] \in V := L^2\left(\left[-0.5,0.5\right]\right).
\]
In \cref{fig:SolutionChannel}, the FE-LF-LTS solution is shown at time $t=1$ for a particular sample $b \approx 0.00589$.
We observe how the wave initiated on the right crosses the channel whose exit acts as a point source in the left rectangle.
\cref{fig:samplesol_2d} shows an overlay of all realizations of the quantity of interest $Q_\ell\left(\omega^{(i,\ell)}\right)$ on levels $\ell = 1,2$ and the MLMC estimate $\widehat{Q}_h^{\ML}$ generated by \cref{alg:MLMCGiles} for a RMS error tolerance $\varepsilon = 5 \cdot 10^{-5}$, which corresponds to a $2 \%$ relative $L^2$ error. 

Next, we compare the cost of computing the same QoI either with or without LTS on each level.
As shown in \cref{tab:2dExampleComparison}, the variances $V_\ell$ with LTS are smaller than with standard LF, which results in a smaller number of samples $N_\ell$ on each level. 
Moreover, the speed-up per sample $C_\ell^{\LF} / C_\ell^{\LTSLF}$ is maximal on the coarsest levels $\ell=0,1$, where the difference between $H_\ell$ and $h^{\operatorname*{f}}$ is greatest.
Here, the optimal values $N_\ell$ are determined by \cref{alg:MLMCGiles} according to \eqref{eq:optimal_Nl} while $V_\ell$ is estimated from \eqref{eq:Vell_est}. 
To estimate $C_\ell$, we use the cost models \eqref{eq:C_l_LF} and \eqref{eq:C_l_LTS} from Sections \ref{sec:localref} and \ref{sec:LTS}, respectively.
For the total cost $\mathcal{C}= \sum_\ell N_\ell C_\ell$, the MLMC method with LTS is about  $6.8$ times faster than MLMC without LTS for the same error tolerance.

\begin{figure}
\centering
\includegraphics[width=0.62\textwidth]{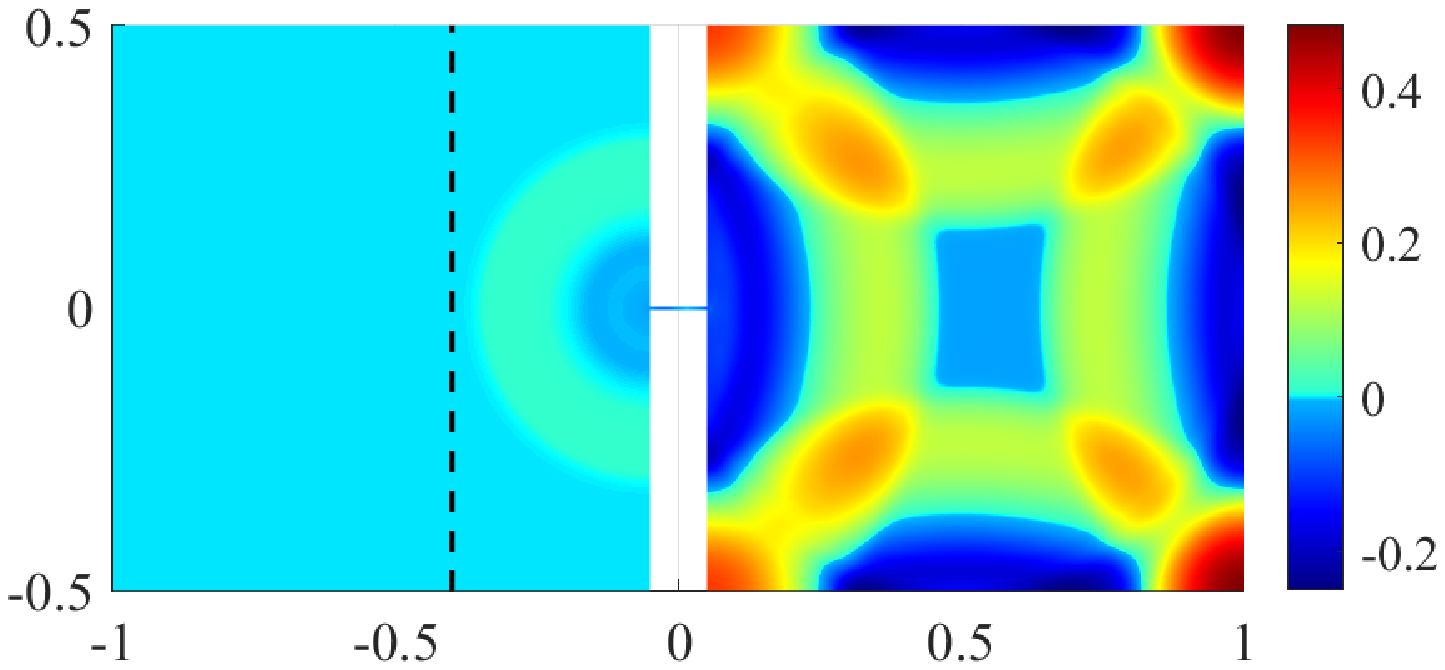}
\includegraphics[width=0.37\textwidth]{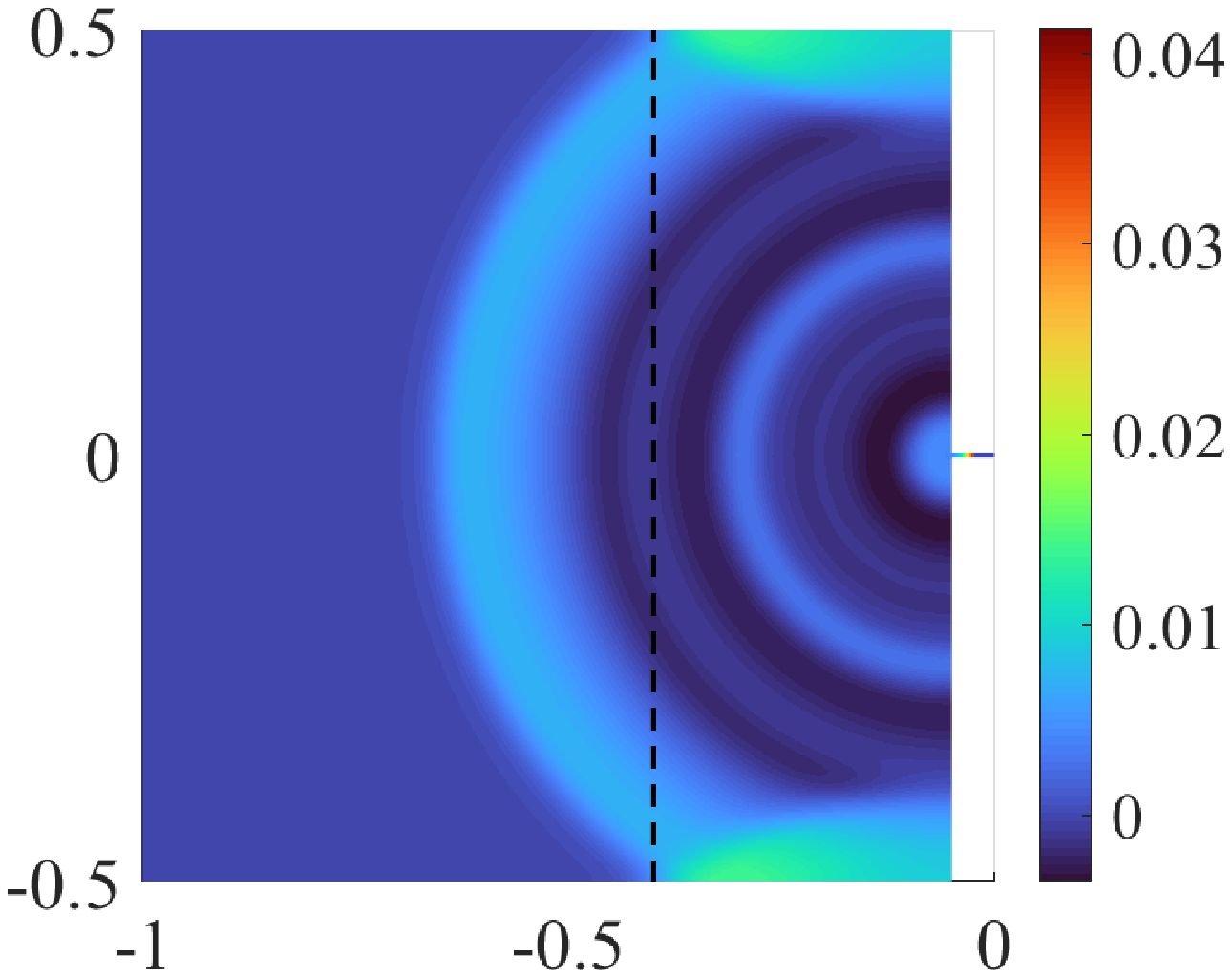}
\caption{Two-dimensional narrow channel. FE solution of the wave equation on $D_{b}$ for $b \approx 0.00589$ at time $t = 1$ with a dashed line at $x=-0.4$. On the right: The same solution restricted to the left rectangle ($x<0$) with an adjusted color map.}
\label{fig:SolutionChannel}
\end{figure}
\begin{figure}
\centering
\includegraphics[width=0.49\textwidth]{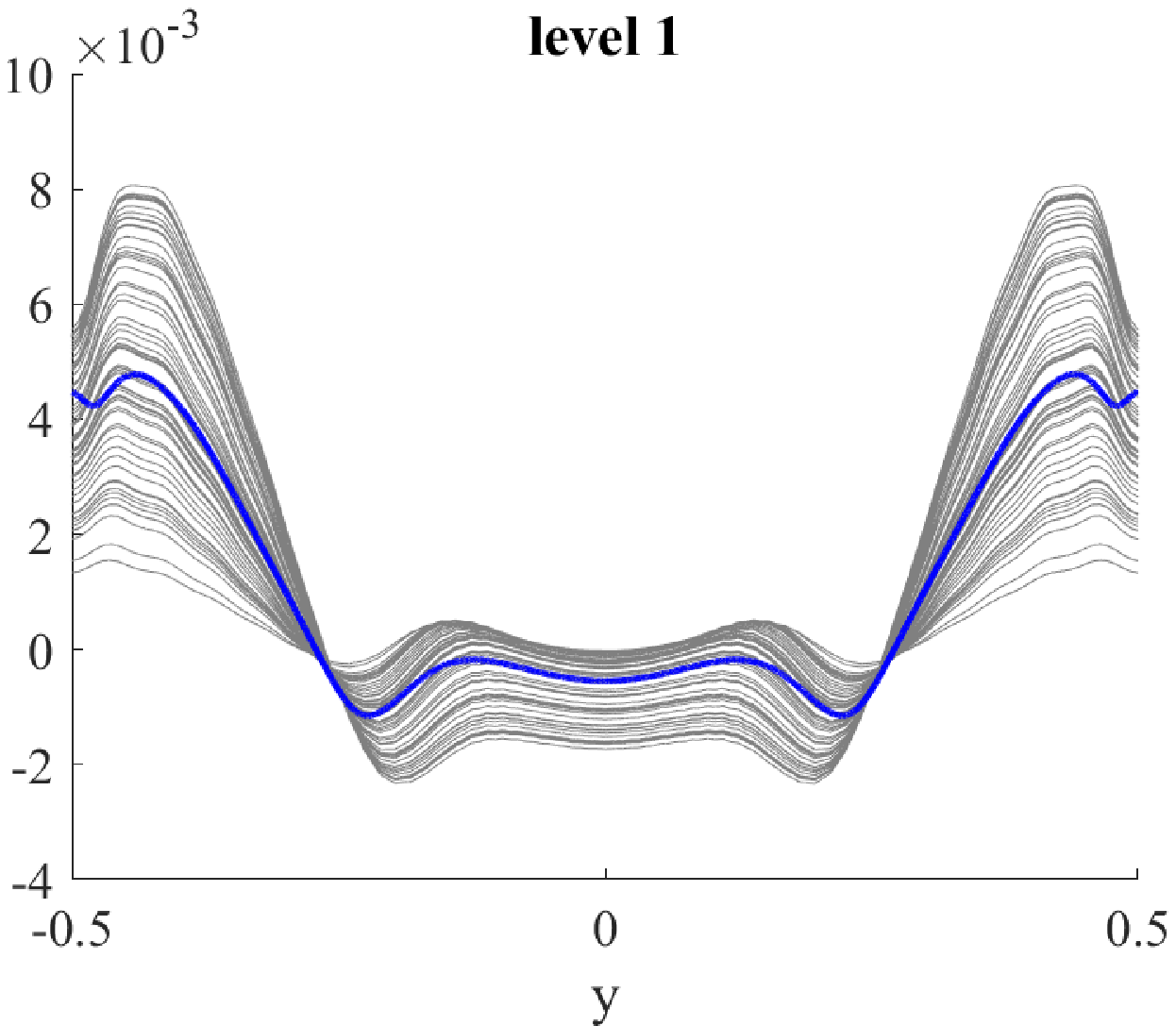}
\includegraphics[width=0.49\textwidth]{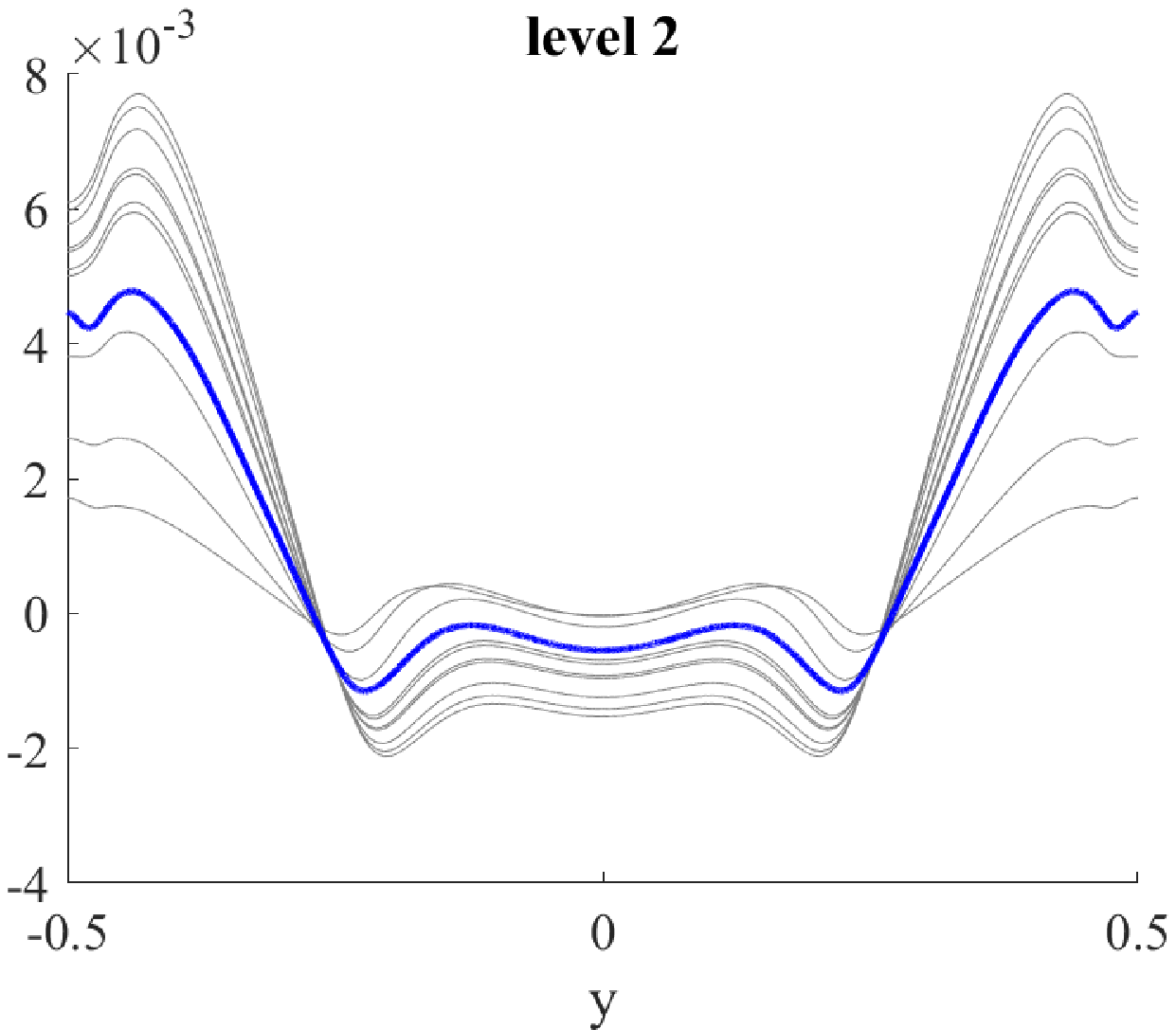}
\caption{Two-dimensional narrow channel. Overlay of solutions $Q_\ell^{(i,\ell)} = u_\ell\left(x,y,T,\omega^{(i,\ell)}\right)$ along the vertical line $x=-0.4$ at fixed time $T=1$ computed by the MLMC \cref{alg:MLMCGiles} with LF-LTS on levels $\ell=1,2$ (grey) together with the final estimate $\widehat{Q}^{\ML}_h$ of $\mathbb{E}[Q]$ (blue).}
\label{fig:samplesol_2d}
\end{figure}

\begin{table}
\caption{Two-dimensional narrow channel. Comparison of using LF-LTS or standard LF to compute MLMC estimate $\widehat{Q}^{\ML}_h$ for RMS error tolerance $\varepsilon = 5\cdot 10^{-5} \approx 0.020 \cdot \|\widehat{Q}^{\ML}_h\|_{L^2(-0.5,0.5)}$.}
\begin{center}
\begin{tabular}{c||c|c|c|c|c}
 & \multicolumn{2}{c|}{variance $V_\ell$} & \multicolumn{2}{c|}{number of samples $N_\ell$} & costs per sample \\
Level & LF-LTS & standard LF & LF-LTS & standard LF & $C_\ell^{\LF} / C_\ell^{\LTSLF}$ \\
\hline\hline
$\ell = 0$ & $1.21 \cdot 10^{-6}$ & $1.22 \cdot 10^{-6}$ & $1266$ & $1455$ & $5.89$ \\
$\ell = 1$ & $6.98 \cdot 10^{-9}$ & $2.04 \cdot 10^{-8}$ & $51$ & $86$ & $7.82$ \\
$\ell = 2$ & $3.21 \cdot 10^{-10}$ & $8.87 \cdot 10^{-10}$ & $10$ & $11$ & $4.54$ \\
$\ell = 3$ & $7.96 \cdot 10^{-11}$ & $2.66 \cdot 10^{-10}$ & $2$ & $3$ & $2.51$
\end{tabular}
\end{center}
\label{tab:2dExampleComparison}
\end{table}

Finally, we compare the cost of the MLMC approach with that using standard MC on any given level. 
For the MLMC method with LTS, \cref{tab:2dExampleDetails} provides not only 
 the number of samples $N_\ell$ and the variance $V_\ell$ with respect to the differences $\Delta Q_\ell$, but also lists the numerical bias on each level, estimated by \eqref{eq:bias_est}, and the variance $V_\ell^{\MC}$ with respect to the quantities $Q_\ell$ itself:
\begin{align}
V_\ell^{\MC} = \mathbb{E}\left[ \left\| Q_\ell - \mathbb{E}\left[Q_\ell\right]\right\|_V^2 \right]
&= \mathbb{E}\left[ \left\| Q_\ell \right\|_V^2 \right] - \left\|\mathbb{E}\left[Q_\ell\right]\right\|_V^2 \nonumber\\
&\approx \frac{1}{N_\ell} \left( \sum_{i=1}^{N_\ell} \left\| Q_\ell^{(i,\ell)} \right\|_V^2 - \frac{1}{N_\ell} \left\| \sum_{i=1}^{N_\ell} Q_\ell^{(i,\ell)} \right\|_V^2 \right).
\label{eq:varQell}
\end{align}
Note that the MLMC \cref{alg:MLMCGiles}, the number of levels is determined adaptively such that the numerical bias is bounded by $\left\|\mathbb{E}\left[Q - Q_\ell\right]\right\|_V^2 < \varepsilon^2/2 = 1.25 \cdot 10^{-9}$.
To bound on level $\ell$ the mean square error for a standard MC estimator, $\widehat{Q}_\ell^{\MC}$,
\[
e\left(\widehat{Q}_\ell^{\MC}\right) = \left(N^{\MC}_\ell\right)^{-1} V^{\MC}_\ell + \left\|\mathbb{E}\left[Q - Q_\ell\right]\right\|_V^2,
\]
by the same error tolerance $\varepsilon^2 = 2.5 \cdot 10^{-9}$ would require $N_\ell^{\MC} = 5221$ or $N_\ell^{\MC} = 9779$ samples on levels $\ell = 3$ or $\ell = 2$, respectively, resulting in a 
$350$ or $80$-fold increase in computational cost over MLMC. On level $\ell=1$, that error tolerance cannot even be reached, as the bias is already larger than $\varepsilon^2$. In summary, even for numerical wave propagation where the coarsest level always ought to resolve the dominating wave-length, the multilevel strategy clearly 
outperforms by far a standard (single-level) MC method on any given grid.

\begin{table}
\caption{
Two-dimensional narrow channel. The variance and bias for LF-LTS applied to MLMC estimate $\widehat{Q}^{\ML}_h$ with RMS error tolerance $\varepsilon = 5\cdot 10^{-5} \approx 0.020 \cdot \|\widehat{Q}^{\ML}_h\|_{L^2(-0.5,0.5)}$.
}
\begin{center}
\begin{tabular}{c||c|c|c|c}
 & number of samples & \multicolumn{2}{c|}{variance} & numerical bias \\
Level & $N_\ell$ & $V_\ell$ (cf. \eqref{eq:Var_Def}) & $V_\ell^{\MC}$ (cf. \eqref{eq:varQell}) &  $\left\|\mathbb{E}\left[Q - Q_\ell\right]\right\|_V^2$ \\
\hline\hline
$\ell = 0$ & $1266$ & $1.21 \cdot 10^{-6}$ & $1.20 \cdot 10^{-6}$ & \\
$\ell = 1$ & $51$ & $6.98 \cdot 10^{-9}$ & $7.24 \cdot 10^{-6}$ & $6.45 \cdot 10^{-9}$ \\
$\ell = 2$ & $10$ & $3.21 \cdot 10^{-10}$ & $1.01 \cdot 10^{-5}$ & $1.47 \cdot 10^{-9}$ \\
$\ell = 3$ & $2$ & $7.96 \cdot 10^{-11}$ & $1.25 \cdot 10^{-5}$ & $9.82 \cdot 10^{-11}$
\end{tabular}
\end{center}
\label{tab:2dExampleDetails}
\end{table}

\section{Concluding remarks}
To overcome the increasingly
stringent bottleneck in MLMC on coarser levels due to locally refined meshes when using explicit time
integration, we have introduced on each level local time-stepping (LTS),
 which adapts the time-step to the local CFL stability constraint.
The combined LTS-MLMC algorithm thus extends the well-known robust and efficient classical MLMC algorithm for uncertainty quantification to wave propagation in complex geometry without
sacrificing explicitness or inherent parallelism. 

In our cost comparison of the MLMC algorithm using  either standard or local time-stepping, 
we distinguish between two typical situations where mesh coarsening cannot occur uniformly across all levels. First, when local refinement
occurs inside a small fixed region of the computational domain, we have proved in \cref{prop:Qeff}
that the asymptotic complexity of the computational effort as $\varepsilon^{-2}$ remains unchanged up to a constant factor as the desired accuracy $\varepsilon\rightarrow 0$. Depending on parameter values, 
however, the combined LTS-MLMC easily achieves a significant (but constant) speed-up both in theory and in
our numerical examples, in fact even more so in higher dimensions. In our one-dimensional computations, for instance, we observe a 30-fold speed-up over MLMC with standard time-stepping.
Second, we have considered graded meshes towards a reentrant corner, which restore the optimal convergence rates of FEM in the presence of singularity. 
In particular, for $L$-shaped domains, we have proved in \cref{cor:totalcostLFgraded} that LTS can even improve the overall asymptotic complexity of MLMC as $\varepsilon\rightarrow 0$. In other words, the smaller the desired error $\varepsilon$, the larger the speed-up of the LTS-MLMC algorithm over standard time integration -- see also \cref{rem:Vell}. 

In our analysis and numerical experiments, we have concentrated on standard continuous piecewise
polynomial finite elements (with mass lumping) for the spatial discretization and 
on the popular leapfrog method for time discreitzation. Our cost estimates nonetheless hold for
other spatial discretizations such as finite difference or discontinuous Galerkin methods, too. Thus, we expect
a similar speed-up when replacing other explicit time integrators, such as Adams-Bashforth or Runge-Kutta methods, by their explicit LTS counterparts \cite{GroteMitkova2010,GroteMitkova13,GroteMehlinMitkova15}. 

Although we have only considered pre-defined sequences of meshes featuring local refinement, the methodology and analysis presented here is also relevant for MLMC combined with adaptive mesh refinement based on a posteriori error estimators \cite{Hoel-vonSchwerin-Szepessy-Tempone 2014, Eigel-Merdon-Neumann 2016, M.B. Giles-Lester-Whittle 2014}. The combined LTS-MLMC approach will also prove
useful for parabolic problems, if the LTS counterpart of
explicit RK-Chebyshev methods \cite{AbdulleGroteSouza21} is used for time integration.

\bibliographystyle{abbrv}

\end{document}